\documentclass[12pt]{article}

\usepackage{amsmath}
\usepackage{amsthm}
\usepackage{amssymb}
\usepackage{amscd}
\usepackage{xspace}
\usepackage{verbatim}
\usepackage{geometry}
\usepackage{esint}
\usepackage{geometry}
\usepackage{fancyhdr}
\usepackage{xargs}
\usepackage{xifthen}
\usepackage{xcolor}

\newcommandx{\Set}[2][2=]{
    \ifthenelse{\isempty{#2}}
        {\left\{ {#1} \right\}}
        {\left\{  {#1}  \, \middle| \, {#2} \right\}}
}

\newtheorem{theorem}{Theorem}[section]
\newtheorem*{theorem*}{Theorem}
\newtheorem{corollary}[theorem]{Corollary}
\newtheorem{lemma}[theorem]{Lemma}
\newtheorem{proposition}[theorem]{Proposition}
\newtheorem*{proposition*}{Proposition}

\theoremstyle{definition}
\newtheorem{definition}[theorem]{Definition}
\theoremstyle{remark}
\newtheorem{remark}{Remark}[section]
\numberwithin{equation}{section}

\newcommand{\ov}{\overline}

\newcommand{\e}{\varepsilon}

\newcommand{\N}{\mathbb N}

\newcommand{\R}{\mathbb{R}}

\newcommand{\er}{\eqref}

 \DeclareMathOperator{\dist}{dist}

\DeclareMathOperator{\cp}{cap}

\DeclareMathOperator{\diam}{diam}

\newcommandx{\norm}[1][1=\cdot]{\left|{#1}\right|}
\newcommandx{\supnorm}[2][1=\cdot,2=]{\nnorm[#1]_{\infty,#2}}
\newcommandx{\nnorm}[1][1=\cdot]{\left|\left|{#1}\right|\right|}

\date{\today}

\begin{document}
\title{Fine properties of Besov functions $B^r_{q,\infty}$ in metric spaces}
\author{Paz Hashash and Arkady Poliakovsky}
\date{\today}
\maketitle

\begin{abstract}
Let $X$ be a metric space and $\mu$ an $s$-regular Ahlfors measure. Let $Y$ be a metric space. We prove that for Besov functions $u \in B^r_{q,\infty}(X,\mu;Y)$, every point is a {\it general average Lebesgue point} of $u$ outside a $\sigma$-finite set with respect to the Hausdorff measure $\mathcal{H}^{s - rq}$. The proof is based on density-type estimates involving Hausdorff measure. In addition, we prove that for functions $u$ in the fractional Sobolev space $W^{r,q}(X,\mu;Y)$, almost every point with respect to $\mathcal{H}^{s - rq}$ is an {\it average Lebesgue point} of $u$. Finally, if $Y$ is also complete, we prove that for $u \in B^r_{q,\infty}(X,\mu;Y)$, almost every point is a {\it Lebesgue point} outside a set of Hausdorff dimension at most $s - rq$.
\end{abstract}

\tableofcontents
\maketitle

\section{Introduction}
The study of fine properties of functions in Sobolev and Besov spaces has played a central role in real and functional analysis, particularly in understanding the behavior of functions at small scales and in non-smooth settings. One of the foundational results in this direction is the classical Lebesgue differentiation theorem, which asserts that almost every point of a locally integrable function is a Lebesgue point. This result connects the integral properties of a function to its pointwise behavior and serves as a cornerstone in modern analysis. One can find proof of Lebesgue differentiation theorem in \cite{EG,F}.

In the context of function spaces that encode smoothness, such as Sobolev, $BV$, or Besov spaces, whose elements are locally integrable and therefore have a set of non-Lebesgue points of measure zero with respect to the underlying measure (Lebesgue measure in Euclidean space), it is natural to ask whether the set of non-Lebesgue points is in fact smaller than merely having measure zero, due to the additional regularity. If so, one may ask how small this set can be. In other words, if a function is not only locally integrable but also possesses some degree of regularity, does it follow that the set of its non-Lebesgue points becomes smaller?

One importance of such knowledge arises in the study of traces of measurable functions, for example in $\mathbb{R}^N$. Suppose we have a measurable function on a Euclidean domain, such as a solution to a partial differential equation, which we need to restrict (trace) to the boundary of the domain or to a surface of dimension less than $N$. In general, this restriction is meaningless unless we know how small the set of non-Lebesgue points is, for example in terms of capacity, Hausdorff measure, or Hausdorff dimension.

For instance, it is known that Sobolev functions $u \in W^{1,p}(\R^N)$ have Lebesgue points outside a set of $\cp_p$-measure zero \cite{EG,FedererZiemer1972,Kinnunen,Ziemer}; a similar result holds for functions in Bessel potential spaces $L^{\alpha,p}(\R^N)$ \cite{AdamsHelberg2012}; for $BV(\R^N)$ functions, every point outside a countably $\mathcal{H}^{N-1}$-rectifiable set is a Lebesgue point \cite{AFP,EG,Ziemer}; for Besov functions $u \in B^r_{q,\infty}(\R^N)$, it is known that the set of points that are not general average Lebesgue points is $\sigma$-finite with respect to $\mathcal{H}^{N-rq}$ \cite{HP2024}; and for fractional Sobolev functions $u \in W^{r,q}(\R^N)$, $r \in (0,1)$, almost every point is an average Lebesgue point with respect to $\mathcal{H}^{N-rq}$ \cite{HP2024}.

Metric measure spaces equipped with a Borel measure satisfying certain regularity assumptions, such as Ahlfors regularity, provide a general framework to extend many classical results of analysis beyond the Euclidean space. In such settings, questions about the local behavior of functions become more subtle due to the lack of linear structure. Nevertheless, a considerable amount of work in the past decades has aimed to extend Sobolev-type theories, including differentiation and capacity estimates, to such spaces, see for example \cite{Shanmugalingam2000,JuhaHeinonen,
HeinonenKoskelaShanmugalingam2015,
Hajlasz2003,KinnunenLatvala2002}.

In this paper, we study the fine properties of functions $u \colon X \to Y$ belonging to the Besov spaces $B^{r}_{q,\infty}(X,\mu;Y)$ and to the fractional Sobolev spaces $W^{r,q}(X,\mu;Y)$, where $X$ is a metric space endowed with an $s$--Ahlfors regular measure $\mu$. The target space $Y$ is, depending on the result, either an arbitrary metric space or a complete metric space. 

By \emph{fine properties} we mean properties that hold almost everywhere with respect to measuring devices that are typically finer than the underlying measure of the space; in the present setting, these finer devices are Hausdorff measures and Hausdorff dimension. Our main focus is the characterization of the set of Lebesgue points of such functions ($W^{r,q}$,$B^{r}_{q,\infty}$).

We now give precise definitions and statements of the main results of this work. 
\\ 
---
\\
Let $X$ be a set and let $\phi$ be a measure on it. Let $(Z,|\cdot|)$ be a normed space. For $p \in (0,\infty)$, the Lebesgue space $L^p(X,\phi;Z)$ is defined to be the space of $\phi$--measurable functions $u : X \to Z$ such that $\int_X |u|^p\, d\phi < \infty$.
If $X$ is also a metric space and $\phi$ is Borel, we define the local space $L^p_{\text{loc}}(X,\phi;Z)$ to be the space of $\phi$--measurable functions $u : X \to Z$ such that $u \in L^p(B,\phi;Z)$ for every closed ball $B \subset X$. In case $Z=\R$, we write $L^p(X,\phi)$, $L^p_{\text{loc}}(X,\phi)$.

The range space \(Y\) of functions $u:X\to Y$ in all our main results will be one of the following two types: a general metric space, or a complete one. For dealing with the notions of general average Lebesgue points and average Lebesgue points, it is sufficient to assume that \(Y\) is a (general) metric space. However, for treating (usual) Lebesgue points, a stronger structure on the target space is required. In results where \(Y\) is a (general) metric space, we use only the metric \(\rho\) to define the function spaces. In contrast, for the study of Lebesgue points we assume that the target space \(Y\) is also complete.

We now give the definitions of Lebesgue points, average Lebesgue points, and general average Lebesgue points.

\begin{definition}[Lebesgue points, Average Lebesgue points, and general average Lebesgue points]
\label{def:lower average Lebesgue point}
Let $X$ be a metric space and let $\mu$ be a Borel measure\footnote{Throughout, by a Borel measure we mean a positive outer measure for which all Borel sets are measurable.} on $X$ such that every closed ball has positive and finite $\mu$-measure. Let $q \in (0,\infty)$. Let $(Y,\rho)$ be a metric space. Assume that $u:X\to Y$ is $\mu$--measurable, and  $F_u \in L^q_{\text{loc}}(X\times X,\mu\times \mu)$, $F_u(z,v):=\rho(u(z),u(v))$. 
\begin{enumerate}
\item A point $x \in X$ is called a \emph{Lebesgue point} of $u$ if there exists an element $u^*(x)\in Y$ such that
\begin{equation}
\lim_{\e \to 0^+}\fint_{B_{\e}(x)}\rho\left(u(z),u^*(x)\right)^q d\mu(z) = 0.
\end{equation}
    \item A point $x \in X$ is called an \emph{average Lebesgue point} of $u$ if
\begin{equation}
\lim_{\e \to 0^+} \fint_{B_{\e}(x)}\fint_{B_{\e}(x)}\rho\left(u(z),u(v)\right)^q d\mu(v)d\mu(z) = 0.
\end{equation} 
\item A point $x \in X$ is called a \emph{general average Lebesgue point} of $u$ if
\begin{equation}
\liminf_{\e \to 0^+} \fint_{B_{\e}(x)}\fint_{B_{\e}(x)}\rho\left(u(z),u(v)\right)^q d\mu(v)d\mu(z) = 0.
\end{equation}   
\end{enumerate}
\end{definition}
Notice that in case $Y=\R^N$ and $q\geq 1$, using triangle inequality and Jensen's inequality, we have for average Lebesgue points $x$
\begin{equation}
\label{eq:equivalence-average Lebesgue}
\lim_{\e \to 0^+} \fint_{B_{\e}(x)} |u(z) - u_{B_{\e}(x)}|^q \, d\mu(z) = 0\, \Longleftrightarrow\, \lim_{\e \to 0^+} \fint_{B_{\e}(x)}\fint_{B_{\e}(x)}|u(z) - u(v)|^q d\mu(v)d\mu(z) = 0.
\end{equation}
That is, the average oscillation of $u$ around $x$ tends to zero as the radius of the ball tends to zero, where
\begin{equation}
u_{B_{\e}(x)}:=\fint_{B_\e(x)}u(v)d\mu(v)=\frac{1}{\mu(B_\e(x))}\int_{B_\e(x)}u(v)d\mu(v).
\end{equation}
Similarly for general average Lebesgue points $x$
\begin{equation}
\label{eq:equivalence-general average Lebesgue}
\liminf_{\e \to 0^+} \fint_{B_{\e}(x)} |u(z) - u_{B_{\e}(x)}|^q \, d\mu(z) = 0 \, \Longleftrightarrow\,  \liminf_{\e \to 0^+} \fint_{B_{\e}(x)}\fint_{B_{\e}(x)}|u(z) - u(v)|^q d\mu(v)d\mu(z) = 0.
\end{equation}
In other words, along some sequence of radii $\varepsilon_k \to 0$, the average oscillation of $u$ on $B_{\varepsilon_k}(x)$ tends to zero.

The following definition of Besov functions on metric spaces is inspired by the definition given in \cite{GPKS2010}, adapted to the setting of functions with values in a metric space:
\begin{definition}[Besov space $B^r_{q,\infty}(X,\mu;Y)$]
\label{def:Besov functions}
Let \(X\) be a metric space, and let \(\mu\) be a Borel measure on \(X\) such that
\(0 < \mu(B) < \infty\) for every closed ball \(B \subset X\).
Let \((Y,\rho)\) be a metric space.
Let \(0< q < \infty\) and \(0 \le r \le 1\). Let \(u : X \to Y\) be a \(\mu\)--measurable function.
We say that \(u \in B^{r}_{q,\infty}(X,\mu;Y)\) if and only if
\(F_u \in L^q_{\mathrm{loc}}(X \times X,\mu \times \mu)\), where
\(F_u(x,z) := \rho(u(x),u(z))\), and
\begin{equation}
\|u\|_{B^{r}_{q,\infty}(X,\mu;Y)}
:= \sup_{t > 0}
\left(
\int_X \fint_{B_t(x)}
\frac{\rho(u(x),u(z))^q}{t^{rq}}
\, d\mu(z)\, d\mu(x)
\right)^{\frac{1}{q}}
< \infty.
\end{equation}
\end{definition}

\begin{remark}
Definition~\ref{def:Besov functions} coincides with the usual definition of Besov functions in the class $B^r_{q,\infty}$ in the case $r<1$. In the case $r=1$, it naturally leads to the definitions of the space of functions of bounded variation $BV$ and of the Sobolev space $W^{1,q}$ between metric spaces. In the case $r=1$ and $q=1$, the space 
\[
BV(X,\mu;Y):=B^{1}_{1,\infty}(X,\mu;Y)
\]
is called the space of functions of bounded variation.  

In the case $r=1$ and $q>1$, the space 
\[
W^{1,q}(X,\mu;Y):=B^{1}_{q,\infty}(X,\mu;Y)
\]
is called the Sobolev space.
\end{remark}

In order to state our main results we need the following concept of regularity of a measure:
\begin{definition}[Ahlfors $s$-regularity]
Let $X$ be a metric space, and let $\mu$ be a Borel measure on $X$. We say that $\mu$ is \emph{Ahlfors $s$-regular} for some $s\in (0,\infty)$ if there exist constants $c, C > 0$ such that for all $x \in  X$ and all $0 < r < \diam(X)$,
\[
c\, r^s \leq \mu(B_r(x)) \leq C\, r^s.
\]
The number $s$ is called the \emph{Ahlfors dimension} of the measure $\mu$.
\end{definition}

Our first result is the following theorem:

\begin{theorem}[General average Lebesgue points of $B^r_{q,\infty}$]
\label{thm:sigma finiteness of limiting average with respect to Hausdorff measure,intro}
Let $X$ be a metric space, and let $\mu$ be an Ahlfors $s$-regular measure on $X$ for some $s\in (0,\infty)$. Let $0\leq r\leq 1$, $0<q<\infty$ be such that $rq\leq s$. Let $Y$ be a metric space. Let $u\in B^r_{q,\infty}(X,\mu;Y)$. 
Then, there exists a $\mathcal{H}^{s-rq}$
$\sigma$-finite set $D\subset X$, such that every point $x\in X\setminus D$ is a general average Lebesgue point of $u$ as defined in Definition \ref{def:lower average Lebesgue point}.
\end{theorem}

This extends previous work on Besov spaces in euclidean spaces \cite{HP2024} to the setting of metric spaces.

Note that by combining the results of
De~Lellis and Otto \cite{DO} and Ghiraldin and Lamy \cite{GL}, given a divergence-free vector field
$u\in B^{\frac{1}{3}}_{3,\infty}(\Omega, S^1)$, where $\Omega\subset\mathbb{R}^2$ is a bounded open domain and
$S^1\subset\mathbb{R}^2$ is the unit circle,
there exists a $\mathcal{H}^{1}$-rectifiable set $D\subset \Omega$ such that every point
$x\in \Omega\setminus D$ is an average Lebesgue point of $u$ as in Definition~\ref{def:lower average Lebesgue point}. 

Therefore, in the general setting of Theorem~\ref{thm:sigma finiteness of limiting average with respect to Hausdorff measure,intro}, we need only the metric structure of Besov functions, without any additional assumptions such as a Euclidean domain, dimension $2$, the divergence-free condition on the vector field, or the requirement that the image lies in the unit sphere. 
However, under these conditions the result mentioned above is stronger: it gives information about average Lebesgue points, not only about general average Lebesgue points; and it provides geometric information about the exceptional set $D$, namely that it is rectifiable. In our general theorem, by contrast, the set $D$ is only $\sigma$-finite, which is weaker than rectifiability.

We now give the definition of fractional Sobolev space:

\begin{definition}[Fractional Sobolev space \(W^{r,q}(X,\mu;Y)\)]
Let \((X,d,\mu)\) be a metric measure space with \(\mu\) Ahlfors \(s\)-regular for some \(s\in(0,\infty)\), and let \((Y,\rho)\) be a metric space. Let \(0<r<1\) and \(0<q<\infty\). For a \(\mu\)-measurable mapping \(u:X\to Y\), define
\[
F_u(x,z):=\rho\bigl(u(x),u(z)\bigr), \qquad (x,z)\in X\times X.
\]
We say that \(u\in W^{r,q}(X,\mu;Y)\) if \(F_u\in L^q_{\mathrm{loc}}(X\times X,\mu\times\mu)\) and 
\[
\|u\|_{W^{r,q}(X,\mu;Y)}
:=\left(\iint_{X\times X}
\frac{\rho\bigl(u(x),u(z)\bigr)^q}{d(x,z)^{rq+s}}
\,d(\mu\times\mu)(x,z)\right)^{1/q}
<\infty.
\]
We say that \(u\in W^{r,q}_{\mathrm{loc}}(X,\mu;Y)\) if $u:X\to Y$ is $\mu$--measurable and 
\(u\in W^{r,q}(B,\mu;Y)\) for every ball \(B\subset X\).
\end{definition}
Our second main result is the following theorem:

\begin{theorem}[Average Lebesgue points of $W^{r,q}$]
\label{thm:hausdorff-dim-D,intro}
Let $X$ be a metric space, and let $\mu$ be an Ahlfors $s$-regular measure on $X$ for some $s\in (0,\infty)$. Let $Y$ be a metric space. Let $0<r<1$, $0<q<\infty$ be such that $rq\leq s$, and let $u\in W^{r,q}_{\text{loc}}(X,\mu;Y)$. Then $\mathcal{H}^{s - rq}$-almost every point is an average Lebesgue point of $u$ as in Definition \ref{def:lower average Lebesgue point}.
\end{theorem}

Our third main result concerns Lebesgue points, rather than merely average or general average points, of Besov functions \(B^r_{q,\infty}(X,\mu;Y)\). In this result, the target space \(Y\) is assumed to be complete.

\begin{theorem}
[Lebesgue points of $B^r_{q,\infty}$]
\label{cor:Hasdorff dimension of D is at most s-rq,intro}
Let $X$ be a metric space, and let $\mu$ be an Ahlfors $s$-regular measure on $X$ for some $s \in (0,\infty)$. Let $Y$ be a complete metric space. Let $0\leq r\leq 1$ and $0< q < \infty$ be such that $rq \leq s$, and let $u \in B^r_{q,\infty}(X,\mu;Y)$. Then there exists a set $\Theta \subset X$ such that the Hausdorff dimension of $\Theta$ is at most $s - rq$, and every point $x \in X \setminus \Theta$ is a Lebesgue point of $u$.
\end{theorem}
\begin{remark}
The result of Theorem~\ref{cor:Hasdorff dimension of D is at most s-rq,intro} holds immediately for functions in the class $W^{r,q}(X,\mu;Y)$, where $0<r<1$ and $0<q<\infty$, because $W^{r,q}(X,\mu;Y)\subset B^r_{q,\infty}(X,\mu;Y)$. See Remark~\ref{inclusion of Sobolev to Besov} for more details.
\end{remark}
 
In \cite{LamyMarconi2024}, X. Lamy and E. Marconi investigated the two-dimensional eikonal equation and showed that its entropy solutions $m:\Omega\subset \R^2\to \R^2$, have the property that the set of non-Lebesgue points has Hausdorff dimension at most~$1$ \cite[Theorem~1.1]{LamyMarconi2024}; see also Remark~1.2 therein. In that remark, the authors point out that the entropy solution $m$ belongs to the Besov space $B^{1/3}_{3,\infty}$, and that the bound on  the Hausdorff dimension of the set of non-Lebesgue points of it can be derived from this regularity. 

Theorem~\ref{cor:Hasdorff dimension of D is at most s-rq,intro} establishes this connection — namely, the connection between Besov regularity $B^r_{q,\infty}$ and the upper bound $s-rq$ for the Hausdorff dimension of the set of non-Lebesgue points — in the greatest possible generality in which Besov spaces can be defined. In particular, this level of generality shows that the above-mentioned connection is a purely metric phenomenon.

In addition, in the above-mentioned remark \cite[Remark~1.2]{LamyMarconi2024}, the authors consider a Hausdorff measure defined via a function of the form $r \mapsto r |\ln r|^{-14}$. In this article, we introduce the notion of logarithmic Hausdorff measure, which will serve as a key tool in the proof of Theorem~\ref{cor:Hasdorff dimension of D is at most s-rq,intro}.

\section{Preliminaries}

We begin by recalling Carathéodory's construction, see also \cite{F,Mattila1995}. Let $X$ be a metric space, let $\mathcal{F}$ be a family of subsets of $X$, and let $\zeta: \mathcal{F} \to [0,\infty]$ be a set function. For $\delta \in (0,\infty]$ and for any subset $A \subset X$, define $\phi_\delta(A)$ to be the infimum of all sums $\sum_{S \in G} \zeta(S)$, where $G \subset \mathcal{F}$ ranges over all countable subfamilies such that $\diam(S) \leq \delta$ for every $S \in G$ and $\bigcup_{S \in G} S \supset A$.

Since for any $0 < \sigma < \delta \leq \infty$ we have $\phi_\delta(A) \leq \phi_\sigma(A)$, it follows that the limit
\begin{equation}
\psi(A) := \lim_{\delta \to 0^+} \phi_\delta(A) = \sup_{\delta > 0} \phi_\delta(A)
\end{equation}
exists.

The set functions $\phi_\delta$ and $\psi$ are outer measures on $X$. The function $\psi$ defines a Borel measure. If the family $\mathcal{F}$ consists of Borel sets, then $\psi$ is a Borel regular measure.

The measure $\psi$ is called the \textbf{result of Carathéodory's construction from $\zeta$ and $\mathcal{F}$}, and the family $\{\phi_\delta\}_{\delta > 0}$ is referred to as the \textbf{approximating measures at size $\delta$}.

In the sequel, we impose the following additional assumptions on the family $\mathcal{F}$ and the function $\zeta$: the family $\mathcal{F}$ consists of all Borel subsets of the metric space $X$, and the function $\zeta$ satisfies the following two additional conditions:

\begin{enumerate}
\item There exists $I\in (0,\infty]$, such that for every $R \in (0,I)$, and $x \in X$, we have
\begin{equation}
\label{eq:property of zeta1}
0<\zeta(B_R(x))<\infty.
\end{equation}

\item There exist constants $\lambda \in (0,\infty)$ and $\gamma \in (0,\infty]$ such that for every $x \in X$ and every radius $0 < R < \gamma$, we have
\begin{equation}
\label{eq:property of zeta}
\zeta(B_{5R}(x)) \leq \lambda\, \zeta(B_R(x)).
\end{equation}
\end{enumerate}

Recall the Vitali covering lemma:
\begin{lemma}[Vitali covering lemma]
\label{lem:Vitali Covering Lemma}
Let $X$ be a metric space, and let $\mathcal{K}$ be a family of non-degenerate closed balls in $X$ such that the radii of all balls in $\mathcal{K}$ are uniformly bounded. Then there exists a pairwise disjoint subcollection $\mathcal{G} \subset \mathcal{K}$ such that
\[
\bigcup_{B\in \mathcal{K}} B \subset \bigcup_{B\in\mathcal{G}} 5B,
\]
where $5B$ denotes the ball with the same center as $B$ but radius multiplied by $5$.
\end{lemma}
For a proof of the Vitali covering lemma, see, for instance, Theorem~1.2 in~\cite{JuhaHeinonen}.

\section{Density-type lemmas involving Carathéodory's construction}
Hereafter, we use the symbol $B_r(x)$ to denote the closed ball centered at $x$ with radius $r\in (0,\infty)$. In the following lemmas, the symbols $\phi_\delta$, $\psi$, $\zeta$, $\mathcal{F}$, and $\lambda$ are as in the previous section. In particular, $\mathcal{F}$ is always assumed to be the family of Borel sets, and $\zeta$ satisfies the additional properties \eqref{eq:property of zeta1} and \eqref{eq:property of zeta}.

Hereafter, we use the symbol $\mathcal{P}(X)$ to denote the power set of a set $X$, that is, the collection of all subsets of $X$.
\begin{lemma}
\label{lem:countability of families of sets}
Let $X$ be a set, and let $F:\mathcal{G}\subset \mathcal{P}(X)\to (0,\infty)$ be a set function.
Assume that there exists a constant $D\in (0,\infty)$ such that for every finite family
$S_1,\dots,S_M\in \mathcal{G}$ of distinct sets, we have
\[
\sum_{i=1}^M F(S_i)\leq D.
\]
Then the set $\mathcal{G}$ is countable.
\end{lemma}

\begin{proof}
For each $n\in\N$, define $\mathcal{G}_n$ to be the collection of sets $S\in \mathcal{G}$ such that
$F(S)\geq 1/n$. Clearly, $\mathcal{G}=\bigcup_{n\in\N}\mathcal{G}_n$.
We claim that each $\mathcal{G}_n$ is finite. Indeed, if $\mathcal{G}_n$ were infinite, then for every
$M\in\N$ we could choose distinct sets $S_1,\dots,S_M\in\mathcal{G}_n$, which would give
\[
D \ge \sum_{j=1}^M F(S_j) \ge \sum_{j=1}^M \frac{1}{n} = \frac{M}{n}.
\]
Letting $M\to\infty$ yields a contradiction with the finiteness of $D$. Therefore,
$\mathcal{G}_n$ is finite for every $n\in\N$. Since $\mathcal{G}$ is a countable union of finite sets, it is countable.
\end{proof}

\begin{lemma}
\label{lem:estimate Hausdorff measure of sup of measures of balls}
Let $X$ be a metric space, let $\Omega \subset X$ be an open
set, and let $\mu$ be a Borel measure on $\Omega$. Let $t
\in (0,\infty)$, $\delta\in \left(0,\min\{I,\gamma\}\right)$, where $I,\gamma$ are the numbers related to properties \eqref{eq:property of zeta}, \eqref{eq:property of zeta1}, and 
\begin{equation}
A\subset \Set{x\in \Omega}[\sup_{\substack{ 0 < \rho < \delta \\ B_\rho(x) \subset \Omega}} \frac{\mu\big(B_\rho(x)\big)}{\zeta\big(B_\rho(x)\big)}>t].
\end{equation}
Then,
\begin{equation}
\label{estimate for 10dleta Haudorff measure of A}
\phi_{10\delta}(A) \leq \frac{\lambda}{t}
\mu(\Omega).
\end{equation}
\end{lemma}

\begin{proof}
If $A=\emptyset$, then the inequality \eqref{estimate for 10dleta Haudorff measure of A} trivially holds, therefore, assume that $A\neq \emptyset$. For each $x \in A$, choose $\rho_x \in (0, \delta)$ such that
$B_{\rho_x}(x) \subset \Omega$ and 
$\mu(B_{\rho_x}(x)) > t \zeta(B_{\rho_x}(x))$.
Let $\mathcal{K}$ be the family of such balls 
$B_{\rho_x}(x)$. Note that the radii of all balls in 
$\mathcal{K}$ are bounded above by $\delta$. Therefore, by Vitali's
covering lemma, Lemma \ref{lem:Vitali Covering Lemma}, there exists a pairwise disjoint subfamily
$\mathcal{G} \subset
\mathcal{K}$ such that
\[
A \subset \bigcup_{x\in A}B_{\rho_x}(x)\subset \bigcup_{B\in\mathcal{G}}5B,
\]
where if $B=B_\rho(x)$, then $5B=B_{5\rho}(x)$. We can assume that $\mu(\Omega)<\infty$, otherwise \eqref{estimate for 10dleta Haudorff measure of A} trivially holds. Therefore, we conclude that $\mathcal{G}$ is countable, since for every choice of a finite collection of distinct balls (and hence pairwise disjoint) from it, say $B_1,\dots,B_N\in \mathcal{G}$, we have by~\eqref{eq:property of zeta1} that $0<\zeta(B_i)<\infty$, and
\begin{equation}
\label{eq:sum of zeta i is finite}
\sum_{i=1}^N 
\zeta(B_i)
\leq \frac{1}{t} \sum_{i=1}^N
\mu\big(B_i\big) = \frac{1}{t}
\mu\left(\bigcup_{i=1}^N B_i\right) \leq
\frac{1}{t} \mu(\Omega)<\infty.
\end{equation}
From \eqref{eq:sum of zeta i is finite} and Lemma \ref{lem:countability of families of sets}, we conclude that $\mathcal{G}$ is countable. Set $\mathcal{G} = \{ B_{\rho_i}(x_i) \}_{i \in \mathbb{N}}$. Hence, using \eqref{eq:property of zeta} 
\[
\phi_{10\delta}(A) \leq \sum_{i=1}^\infty 
\zeta(B_{5\rho_i}(x_i))
\leq \lambda\sum_{i=1}^\infty 
\zeta(B_{\rho_i}(x_i)) \leq \frac{\lambda}{t} \sum_{i=1}^\infty
\mu\big(B_{\rho_i}(x_i)\big) = \frac{\lambda}{t}
\mu\left(\bigcup_{i=1}^\infty B_{\rho_i}(x_i)\right) \leq
\frac{\lambda}{t} \mu(\Omega).
\]
\end{proof}

\begin{lemma}
\label{lem:the upper density lemma for a sequence of Radon measures}
Let $X$ be a metric space. For every $n \in \mathbb{N}$, let $\mu_n$ be a Borel measure on
$X$. Let $t \in [0, \infty)$, and let  $r_n \downarrow 0$ be a sequence
of positive numbers. Let 
\begin{equation}
\label{eq:bound from below on the generalized upper density} 
A\subset \Set{x\in X}[\liminf_{n \to \infty} \left( \sup_{0 < \rho < r_n}
\frac{ \mu_n\big(B_\rho(x)\big)}{ \zeta\big(B_\rho(x)\big)} \right)
 \geq t ].
\end{equation} 
Then, for every open set $\Omega$ containing $A$ we get
\begin{equation}
\label{eq:liminf bounds from above psi A}
\lambda\left( \liminf_{n \to \infty} \mu_n(\Omega)\right) \geq t \psi(A).
\end{equation}
\end{lemma}

\begin{proof}  
We can assume that $t>0$, otherwise \eqref{eq:liminf bounds from above psi A} trivially holds. Fix $\beta \in (0,1)$. For each $n \in \mathbb{N}$, define
\begin{equation}
A_{n,\beta} := \Set{ x \in A }[\sup_{0 < \rho < r_n } \frac{\mu_n\big(B_\rho(x)\big)}{\zeta\big(B_\rho(x)\big)}
 > t \beta ].
\end{equation}
By assumption \eqref{eq:bound from below on the generalized upper density}, we have
\begin{equation}
\label{eq:psi of A minus the liminf of Ajbeta is zero}
A =\bigcup_{k=1}^\infty \bigcap_{j=k}^\infty A_{j,\beta}.
\end{equation}
Indee, by definition, we have $A \supset \bigcup_{k=1}^\infty \bigcap_{j=k}^\infty A_{j,\beta}$.
For the reverse inclusion, let $x\in A$ and assume, by contradiction, that 
$x\notin \bigcup_{k=1}^\infty \bigcap_{j=k}^\infty A_{j,\beta}$.
Then for every $k\in\N$ we have 
$x\notin \bigcap_{j=k}^\infty A_{j,\beta}$.
Hence, there exists a sequence of indices $\{j_l\}_{l\in\N}$ with $j_l\to\infty$ as $l\to\infty$ such that 
$x\notin A_{j_l,\beta}$, for every $l\in\N$.
Therefore, we obtain
\begin{equation}
\label{eq:bound from below on the generalized upper density1} 
\liminf_{n \to \infty} \left( \sup_{0 < \rho < r_n}
\frac{\mu_n\big(B_\rho(x)\big)}{\zeta\big(B_\rho(x)\big)}  \right)
\leq \liminf_{l \to \infty} \left( \sup_{0 < \rho < r_{j_l} }
\frac{\mu_{j_l}\big(B_\rho(x)\big)}{\zeta\big(B_\rho(x)\big)} \right)
 \leq t\beta<t.
\end{equation} 
This contradicts \eqref{eq:bound from below on the generalized upper density}. It proves \eqref{eq:psi of A minus the liminf of Ajbeta is zero}. Since $\psi$ is a Borel regular measure and therefore has the monotonicity property (even when the sets are not $\psi$-measurable), we obtain from \eqref{eq:psi of A minus the liminf of Ajbeta is zero}
\begin{equation}
\label{measure of intersection from n1}
\psi(A) = \psi \left( \bigcup_{k=1}^\infty \bigcap_{j=k}^\infty A_{j,\beta} \right)
= \lim_{k \to \infty} \psi\left( \bigcap_{j=k}^\infty A_{j,\beta} \right).
\end{equation}
Applying Lemma~\ref{lem:estimate Hausdorff measure of sup
of measures of balls} (with $\Omega=X$), for every $n,m\in \N$ such that $m \geq n$, we obtain
\begin{multline}
\label{eq:estimate for K}
\phi_{10r_m} \left( \bigcap_{j=n}^\infty A_{j,\beta} \right)
\leq \phi_{10r_m}(A_{m,\beta})
\\
\leq \phi_{10r_m}\left(\Set{x \in X}[ \sup_{0 < \rho < r_m } \frac{\mu_m\big(B_\rho(x)\big)}{\zeta\big(B_\rho(x)\big)}
 > t \beta ]\right)
\leq \frac{\lambda}{t \beta}\mu_m(X).
\end{multline}
Taking the lower limit as $m \to \infty$ in \eqref{eq:estimate for K} and using that $\lim_{\delta\to 0^+}\phi_\delta=\psi$, we get
\begin{equation}
\label{eq:Haudorff measure of intersection from n}
\psi \left( \bigcap_{j=n}^\infty A_{j,\beta} \right) \leq
\frac{\lambda}{t\beta} \left( \liminf_{m \to \infty} \mu_m(X) \right).
\end{equation}
Taking the limit as $n\to \infty$ in \eqref{eq:Haudorff measure of intersection from n} and using \eqref{measure of intersection from n1}, we get 
\begin{equation}
\label{eq:final estimate with beta}
\psi(A)=\lim_{n \to \infty} \psi \left( \bigcap_{j=n}^\infty A_{j,\beta} \right)
\leq \frac{\lambda}{t\beta} \left( \liminf_{m \to \infty} \mu_m(X) \right).
\end{equation}
Since $\beta\in (0,1)$ is arbitrary, taking the limit as $\beta\to 1^-$ in \eqref{eq:final estimate with beta}, we get
\begin{equation}
\label{eq:final estimate with beta1}
\psi(A)
\leq \frac{\lambda}{t} \left( \liminf_{m \to \infty} \mu_m(X) \right).
\end{equation}
If $\Omega \subset X$ is an open set containing $A$, we may replace the measures $\mu_m$ by the measures $\mu_m^{\Omega}$ defined by $\mu_m^{\Omega}(E) := \mu_m(\Omega \cap E)$. Note that, since $\Omega$ is open and $A\subset \Omega$, the condition \eqref{eq:bound from below on the generalized upper density} still holds for these modified measures. So we get for every open set $\Omega$ with $A\subset \Omega\subset X$
\[
t \psi(A) \leq \lambda \left(\liminf_{m \to \infty} \mu^{\Omega}_m(X)\right)=\lambda \left(\liminf_{m \to \infty} \mu_m(\Omega)\right).
\]
\end{proof}
\begin{remark}
\label{rem:densities}
Recall that for a measure $\mu$ on a metric space $X$, and for $s \in (0,\infty)$, the upper and lower $s$-dimensional densities of $\mu$ at $x \in X$ are respectively defined, up to multiplication by a positive constant, by
\begin{equation}
\Theta^*_{s}(\mu,x) := \limsup_{\rho \to 0^+} \frac{\mu\left(B_\rho(x)\right)}{\rho^{s}}, \quad 
\Theta_{*s}(\mu,x) := \liminf_{\rho \to 0^+} \frac{\mu\left(B_\rho(x)\right)}{\rho^{s}}.
\end{equation}
If $\Theta^*_{s}(\mu,x) = \Theta_{*s}(\mu,x)$, then their common value is denoted by $\Theta_{s}(\mu,x)$.

We refer to the inequality in \eqref{eq:bound from below on the generalized upper density} as a \emph{density-type condition}, as it generalizes the usual upper density in the case where all measures in the sequence $\{\mu_n\}_{n\in\N}$ are equal to the same measure and $\zeta(E):=\left(\diam(E)/2\right)^s$.
\end{remark}
The following lemma is a generalization of Lemma \ref{lem:the upper density lemma for a sequence of Radon measures}. 
\begin{lemma}
\label{lem:the upper density lemma for a sequence of Radon measures1}
Let $X$ be a metric space. For each $n \in \mathbb{N}$, let $\mu_n$ be a Borel measure on $X$. Let $\{r_n\}_{n\in\N}$ be a sequence of positive numbers with $\lim_{n \to \infty} r_n = 0$. Assume the existence of an open set $\Omega \subset X$, a set $\Theta\subset X$ with $\psi(\Theta)=0$, and $\psi$-measurable function $g:\Omega\to [0,\infty]$, such that
\begin{equation}
\label{eq:bound from below on the generalized upper density(1)}
\Omega\setminus \Theta\subset \Set{x\in X}[\liminf_{n \to \infty} \left( \sup_{0 < \rho < r_n}
\frac{\mu_n\big(B_\rho(x)\big)}{\zeta\big(B_\rho(x)\big)}  \right)
 \geq g(x)].
\end{equation}
Then,
\begin{equation}
\label{eq:generalized density theorem for non-negative functions6}
\lambda \left( \liminf_{n \to \infty} \mu_n(\Omega)\right) \geq \int_{\Omega}g(x) \, d\psi(x).
\end{equation}
\end{lemma}

\begin{proof}
The idea is to first prove inequality~\eqref{eq:generalized density theorem for non-negative functions6} for simple functions $h$ that are bounded above by $g$, and then use the definition of the Lebesgue integral to derive inequality~\eqref{eq:generalized density theorem for non-negative functions6}.  Let $h = \sum_{i=1}^k \beta_i \chi_{B_i}$ be a function where $k \in \mathbb{N}$, the $\beta_i$ are positive numbers, the $B_i \subset \Omega$ are pairwise disjoint $\psi$-measurable sets, and $\chi_{B_i}$ denotes the characteristic function of $B_i$. Assume that $g(x) \geq h(x)$ for $\psi$-almost every $x \in \Omega$. In particular, for each $i = 1, \dots, k$, we have $g(x) \geq \beta_i$ for $\psi$-almost every $x \in B_i$. 

We may assume without loss of generality that $\liminf_{n \to \infty} \mu_n(\Omega) < \infty$, otherwise the inequality \eqref{eq:generalized density theorem for non-negative functions6} trivially holds. Therefore, by Lemma~\ref{lem:the upper density lemma for a sequence of Radon measures}, we obtain
\[
\beta_i \psi(B_i) \leq \lambda \left( \liminf_{n \to \infty} \mu_n(\Omega) \right) < \infty.
\]
Since $\psi$ is a Borel regular measure, $B_i$ is a $\psi$-measurable set and $\psi(B_i)< \infty$, for every $\epsilon > 0$, there exists a closed set $C_i \subset B_i$ such that $\psi(B_i \setminus C_i) < \frac{\epsilon}{\beta_i k}$. Since $C_i\subset B_i$, $C_i$ is $\psi$-measurable and $\psi(C_i)<\infty$, we obtain
\begin{equation}
\label{eq:approximating Bi from below by a closed set}
\psi(B_i)-\frac{\epsilon}{\beta_i k}< \psi(C_i).
\end{equation}
Note that since the sets $B_i$ are pairwise disjoint and $C_i \subset B_i$, the closed sets $C_i$ are also pairwise disjoint. As $X$ is a metric space, it is normal; hence, there exist pairwise disjoint open sets $U_i$ such that $C_i \subset U_i\subset \Omega$. Applying Lemma~\ref{lem:the upper density lemma for a sequence of Radon measures} again, we get
\begin{equation}
\label{eq:usage of previous lemma for Ci}
\lambda \left( \liminf_{n \to \infty} \mu_n(U_i) \right) \geq \beta_i \psi(C_i).
\end{equation}
Therefore, since $\mu_n$ is a Borel measure, we get from \eqref{eq:usage of previous lemma for Ci} and \eqref{eq:approximating Bi from below by a closed set}
\begin{multline}
\lambda \left( \liminf_{n \to \infty} \mu_n(\Omega) \right) 
\geq \lambda \left( \liminf_{n \to \infty} \mu_n\left( \bigcup_{i=1}^k U_i \right) \right)
= \lambda \left( \liminf_{n \to \infty} \sum_{i=1}^k \mu_n(U_i) \right) \\
\geq \sum_{i=1}^k \lambda \left( \liminf_{n \to \infty} \mu_n(U_i) \right)
\geq \sum_{i=1}^k \beta_i \psi(C_i)
\geq \sum_{i=1}^k \beta_i \left( \psi(B_i) - \frac{\epsilon}{\beta_i k} \right) 
\\
=\sum_{i=1}^k \beta_i \psi(B_i)-\epsilon= \int_{\Omega} h(x) \, d\psi(x) - \epsilon.
\end{multline}
Since $\epsilon > 0$ is arbitrary, we conclude that
\begin{equation}
\lambda \left( \liminf_{n \to \infty} \mu_n(\Omega) \right) \geq \int_{\Omega} h(x) \, d\psi(x).
\end{equation}
Therefore, by the definition of the Lebesgue integral (approximating $g$ from below by simple functions), we obtain \eqref{eq:generalized density theorem for non-negative functions6} for the $\psi$--measurable function $g$.
\end{proof}

\section{Fine properties of general functions with respect to the Carathéodory resulting measure $\psi$}

The following lemma is a consequence of Lemma~\ref{lem:estimate Hausdorff measure of sup of measures of balls} and is independent of the results proved subsequently. We will use it in the proof of Theorem~\ref{thm:Besov constant controls double average integral}.

\begin{lemma}
\label{lem:the upper density lemma for a sequence of Radon measures:limsupversion}
Let $X$ be a metric space. For every $n \in \mathbb{N}$, let $\mu_n$ be a Borel measure on
$X$. Let 
\begin{equation}
\label{eq:bound from below on the generalized upper density:limsupversion} 
A\subset \Set{x\in X}[\lim_{\delta\to 0^+}\limsup_{n \to \infty} \left( \sup_{0 < \rho < \delta}
\frac{ \mu_n\big(B_\rho(x)\big)}{ \zeta\big(B_\rho(x)\big)} \right)
 >0].
\end{equation} 
Assume the existence of an open set $\Omega$ containing $A$ such that 
\begin{equation}
\label{eq:convergence zero of mu n} 
\sum\limits_{n=1}^{\infty}\mu_n(\Omega)<\infty.
\end{equation} 
Then we have
\[
\psi(A)=0.
\]
\end{lemma}
\begin{proof}
For $\beta>0$, $\delta>0$, and $n \in \mathbb{N}$, define
\begin{equation}
A_{n,\beta,\delta} := \Set{ x \in A }[\sup_{\substack{0 < \rho < \delta \\ B_\rho(x) \subset \Omega}} \frac{\mu_n\big(B_\rho(x)\big)}{\zeta\big(B_\rho(x)\big)}
 > \beta ].
\end{equation}
Define 
\begin{equation}
\label{eq:def of O}
O:=\bigcup_{\beta>0}\bigcap_{\delta>0}\bigcap_{k=1}^\infty\bigcup_{j=k}^\infty A_{j,\beta,\delta}.
\end{equation}
We will show that the property \eqref{eq:convergence zero of mu n} gives $\psi\left(O\right) = 0$, and the property \eqref{eq:bound from below on the generalized upper density:limsupversion} gives $A=O$. Therefore, $\psi(A)=\psi(O)=0$.

We prove $\psi(O)=0$. By Lemma \ref{lem:estimate Hausdorff measure of sup of measures of balls}
\begin{equation}
\label{estimate for 10dleta Haudorff measure of Anuiuiu}
\phi_{10\delta}(A_{n,\beta,\delta}) \leq \frac{\lambda}{\beta}
\mu_n(\Omega).
\end{equation}
Therefore, by $\sigma$-subaditivity of $\phi_{10\delta}$, estimate \eqref{estimate for 10dleta Haudorff measure of Anuiuiu}, and assumption \eqref{eq:convergence zero of mu n}, we obtain for $k\in\N$
\begin{equation}
\label{estimate for 10dleta Haudorff measure of Angguj}
\phi_{10\delta}\left(\bigcup_{j=k}^\infty A_{j,\beta,\delta}\right)\leq\sum\limits_{j=k}^\infty\phi_{10\delta}\left(A_{j,\beta,\delta}\right) \leq\frac{\lambda}{\beta}\sum\limits_{j=k}^{\infty}\mu_j(\Omega)<\infty.
\end{equation}
Thus, by \eqref{estimate for 10dleta Haudorff measure of Angguj}, we get
\begin{equation}
\label{estimate for 10dleta Haudorff measure of Anggujkklkl}
\phi_{10\delta}\left(\bigcap_{\xi>0}\bigcap_{k=1}^\infty\bigcup_{j=k}^\infty A_{j,\beta,\xi}\right)\leq \phi_{10\delta}\left(\bigcap_{k=1}^\infty\bigcup_{j=k}^\infty A_{j,\beta,\delta}\right) \leq\lim_{k\to\infty}\frac{\lambda}{\beta}\sum\limits_{j=k}^{\infty}\mu_j(\Omega)=0.
\end{equation}
Then, letting $\delta\to 0^+$ and using the definition $\lim_{\delta\to 0^+}\phi_\delta=\psi$, we obtain
\begin{equation}
\label{estimate for 10dleta Haudorff measure of Anggujkklkljh}
\psi\left(\bigcap_{\xi>0}\bigcap_{k=1}^\infty\bigcup_{j=k}^\infty A_{j,\beta,\xi}\right)=0.
\end{equation}
Therefore, since $A_{j,\beta',\xi}\subset A_{j,\beta,\xi}$ for $0<\beta\leq \beta'<\infty$, and fixed $j,\xi$, we get
\begin{equation}
\label{estimate for 10dleta Haudorff measure of Anggujkklkljhiuiuiuui}
\psi(O)=\psi\left(\bigcup_{\beta>0}\bigcap_{\xi>0}\bigcap_{k=1}^\infty\bigcup_{j=k}^\infty A_{j,\beta,\xi}\right)=\psi\left(\bigcup_{i=1}^\infty\bigcap_{\xi>0}\bigcap_{k=1}^\infty\bigcup_{j=k}^\infty A_{j,1/i,\xi}\right)=0.
\end{equation}

We prove $A=O$. By the definition of $O$, we get $O\subset A$. For the opposite inclusion, let $x\in A$.  Assume by contradiction that $x\notin O$. Then $x\notin \bigcap_{\xi>0}\bigcap_{k=1}^\infty\bigcup_{j=k}^\infty A_{j,\beta,\xi}$ for every $\beta>0$. Then, there exist $\xi,k$ such that $x\notin \bigcup_{j=k}^\infty A_{j,\beta,\xi}$. Therefore, for every natural number $j\geq k$
\begin{equation}
\sup_{\substack{0 < \rho < \xi \\ B_\rho(x) \subset \Omega}} \frac{\mu_j\big(B_\rho(x)\big)}{\zeta\big(B_\rho(x)\big)}
\leq \beta.
\end{equation}
Therefore,
\begin{multline}
\lim_{\delta \to 0^+}\limsup_{j\to \infty}\left(\sup_{0 < \rho < \delta} \frac{\mu_j\big(B_\rho(x)\big)}{\zeta\big(B_\rho(x)\big)}\right)
= 
\lim_{\delta \to 0^+}\limsup_{j\to \infty}\left(\sup_{\substack{0 < \rho < \delta \\ B_\rho(x) \subset \Omega}} \frac{\mu_j\big(B_\rho(x)\big)}{\zeta\big(B_\rho(x)\big)}\right)
\\
\leq \limsup_{j\to \infty}\left(\sup_{\substack{0 < \rho < \xi \\ B_\rho(x) \subset \Omega}} \frac{\mu_j\big(B_\rho(x)\big)}{\zeta\big(B_\rho(x)\big)}\right)
\leq \beta.
\end{multline}
Since $\beta>0$ is arbitrarily small, we get that 
\begin{equation}
\lim_{\delta \to 0^+}\limsup_{j\to \infty}\left(\sup_{0 < \rho < \delta} \frac{\mu_j\big(B_\rho(x)\big)}{\zeta\big(B_\rho(x)\big)}\right)=0.
\end{equation}
This contradicts \eqref{eq:bound from below on the generalized upper density:limsupversion}.
This completes the proof.
\end{proof}

The following set-theoretic lemma will help us easily recognize the Borel measurability of functions that appear in the sequel.
\begin{lemma}
\label{lem:uper semicontinuity of h}
Let $X$ be a metric space. Let $\Sigma$ denote the family of all closed balls in $X$, and let $F:\Sigma\to [0,\infty)$ be a set function satisfying the following two properties: for $A,B\in \Sigma$, if $A\subset B$, then $F(A)\leq F(B)$; and for every $x\in X$ and $r_0\in (0,\infty)$,
\begin{equation}
\label{eq:right-continuous property}
\lim_{r\downarrow r_0}F(B_r(x))=F(B_{r_0}(x)).
\end{equation}
Fix $r_0\in (0,\infty)$ and define
\begin{equation}
h(x):=F(B_{r_0}(x)).
\end{equation}
Then $h$ is upper semicontinuous and therefore Borel measurable. 
\end{lemma}

\begin{proof}
Let $\{x_j\}_{j\in\N}\subset X$ and $x_0\in X$ such that $x_j\to x_0$ as $j\to \infty$. Fix $\delta>0$. Then there exists $j_0\in\N$ such that 
\[
B_{r_0}(x_j)\subset B_{r_0+\delta}(x_0)
\quad \text{for every } j\geq j_0.
\]
Hence, by the definition of $h$ and the monotonicity of $F$, we obtain
\begin{equation}
\label{eq:mono of F}
\limsup_{j\to \infty}h(x_j)
=\limsup_{j\to \infty}F(B_{r_0}(x_j))
\leq F(B_{r_0+\delta}(x_0)).
\end{equation}
Letting $\delta\downarrow 0$ and using the assumed right-continuity of $r\mapsto F(B_r(x_0))$, we conclude that
\[
\limsup_{j\to \infty}h(x_j)\leq h(x_0).
\]
Thus, $h$ is upper semicontinuous. Any upper semicontinuous function $f$ from topological space $Z$ to $\R$ is Borel measurable, since upper semicontinuity implies that the upper level sets 
$$
\Set{z\in Z}[f(z)\ge a]=f^{-1}\left([a,\infty)\right)
$$
are closed for every $a\in\mathbb{R}$, and hence Borel. Therefore, $f$ is a Borel function.
\end{proof}

\begin{corollary}
\label{cor:double average is Borel}
Let $X$ be a metric space. Let $\mu$ be a Borel measure on $X$ such that $0<\mu(B)<\infty$ for every closed ball $B=B_\sigma(x)$, $x\in X$, $\sigma\in (0,\infty)$. Let $0\leq W\in L^1_{\mathrm{loc}}(X\times X,\mu\times \mu)$. Define
\begin{equation}
\label{eq:double average is Borel}
G(x):=\fint_{B_{\sigma}(x)}\fint_{B_{\sigma}(x)}W(z,y)\,d\mu(z)\,d\mu(y).
\end{equation}
Then $G$ is Borel measurable.
\end{corollary}

\begin{proof}
Define
\[
F_1(B_{\sigma}(x))
:=\int_{B_{\sigma}(x)}\int_{B_{\sigma}(x)}W(z,y)\,d\mu(z)\,d\mu(y),
\qquad
F_2(B_{\sigma}(x))
:=\mu(B_{\sigma}(x)).
\]
Both $F_1$ and $F_2$ satisfy the two assumptions of Lemma~\ref{lem:uper semicontinuity of h}; the property \eqref{eq:right-continuous property} holds since the balls $B_{\sigma}(x)$ are closed, $\mu$ is finite on balls and $W$ is locally integrable. More precisely,
\begin{equation}
\lim_{\sigma\downarrow \sigma_0}F_1(B_\sigma(x))=\int_{\bigcap\limits_{\sigma>\sigma_0}B_{\sigma}(x)\times B_{\sigma}(x)}W(z,y)\,d\left(\mu\times \mu\right)(z,y)=F_1(B_{\sigma_0}(x));
\end{equation}
\begin{equation}
\lim_{\sigma\downarrow \sigma_0}F_2(B_\sigma(x))=\mu\left(\bigcap\limits_{\sigma>\sigma_0}B_{\sigma}(x)\right)=F_2(B_{\sigma_0}(x)).
\end{equation}
Hence, the functions
\[
x\mapsto F_1(B_{\sigma}(x))
\quad \text{and} \quad
x\mapsto F_2(B_{\sigma}(x))
\]
are upper semicontinuous and therefore Borel measurable. As $\mu(B_{\sigma}(x))>0$ for every $x\in X$, we conclude that
\[
G(x)=\frac{F_1(B_{\sigma}(x))}{F_2(B_{\sigma}(x))^2}
\]
is Borel measurable.
\end{proof}

The following theorem provides an abstract framework for establishing certain fine properties of a general function $W$ with respect to any measure $\psi$ obtained via Carathéodory’s construction. In the sequel, we will choose $W(z,y):=\rho(u(z),u(y))^q$ and take $\psi$ to be the Hausdorff measure in the case of Besov functions $u\in B^r_{q,\infty}$.
\begin{theorem}
\label{thm:Besov constant controls double average integral}
Let $X$ be a metric space and let $\Omega \subset X$ be an open subset. 
Let $\mu$ be an Ahlfors $s$-regular measure on $\Omega$ for some $s \in (0,\infty)$. Consider any sequences $(\varepsilon_n)_{n\in\mathbb{N}}$ and $(\eta_n)_{n\in\mathbb{N}}$ with 
$\varepsilon_n > 0$, $\eta_n > 0$ and $\lim_{n\to\infty} \varepsilon_n = 0$, and let 
$W \in L^1_{\mathrm{loc}}(\Omega \times \Omega, \mu \times \mu)$ with $W \geq 0$. Finally,
assume that $\zeta$ satisfies, in addition to \eqref{eq:property of zeta} and \eqref{eq:property of zeta1} the following:
\begin{enumerate}
\item There exist constants $\lambda' \in (0,\infty)$ and $\gamma' \in (0,\infty]$ such that for every $x \in X$ and every radius $0 < R < \gamma'$, we have
\begin{equation}
\label{eq:property of zeta3}
\zeta(B_{R}(x)) \leq \lambda'\, \zeta(B_{R}(y))\quad\forall y\in B_{R/2}(x).
\end{equation}

\item For every $R \in (0,\infty)$ 
\begin{equation}
\label{eq:property of zeta5}
f_R(x):=\zeta(B_R(x)) \quad\text{is $\mu$-measurable in $X$}.
\end{equation}
\end{enumerate}
Then:
\\
1. There exists a constant
$C(\zeta,\mu)$, depending on $\zeta,\mu$ only, such that
\begin{multline}
\label{eq:Besov constant is more that integral of lower limiting double integralgen}
\liminf\limits_{n\to\infty}\left(\eta_n\int_\Omega\frac{1}{\e^s_n}\zeta\big(B_{\e_n}(x)\big)\Bigg\{\fint_{B_{\e_n}(x)}\chi_\Omega(
z) W(z,x)d\mu(z)\Bigg\}d\mu(x)\right)
\\
\geq C(\zeta,\mu)\int_\Omega \Bigg\{\liminf\limits_{n\to
\infty}\eta_n\fint_{B_{\e_n/2}(x)}\fint_{B_{\e_n/2}(x)} W(z,y)d\mu(z)d\mu(y)\Bigg\}d\psi(x)\,.
\end{multline}
\\
2. If the left hand side of \eqref{eq:Besov constant is more that integral of lower limiting double integralgen} is finite, then the set $D\subset \Omega$ defined by
\begin{equation}
D=\Set{x\in \Omega}[\liminf\limits_{n\to
\infty}\left(\eta_n\fint_{B_{\e_n/2}(x)}\fint_{B_{\e_n/2}(x)} W(z,y)d\mu(z)d\mu(y)\right)>0]
\end{equation}
is $\sigma$-finite with respect to $\psi$, and for every $x\in \Omega\setminus D$, we have
\begin{equation}
\label{eq:Lebesgue points for Wgen}
\liminf\limits_{n\to
\infty}\left(\eta_n\fint_{B_{\e_n/2}(x)}\fint_{B_{\e_n/2}(x)} W(z,y)d\mu(z)d\mu(y)\right)=0.
\end{equation} 
\\
3. If we have
\begin{equation}
\label{eq:Besov constant is more that integral of lower limiting double integralgensum}
\sum\limits_{n=1}^{\infty}\left(\eta_n\int_\Omega\frac{1}{\e^s_n}\zeta\big(B_{\e_n}(x)\big)\Bigg\{\fint_{B_{\e_n}(x)}\chi_\Omega(
z) W(z,x)d\mu(z)\Bigg\}d\mu(x)\right)
<\infty\,,
\end{equation}
then the set $A\subset \Omega$ defined by
\begin{equation}
\label{eq:def of A}
A:=\Set{x\in \Omega}[\limsup_{n\to\infty}\left(\eta_n\fint_{B_{\e_n/2}(x)}\fint_{B_{\e_n/2}(x)} W(z,y)d\mu(z)d\mu(y)\right)>0]
\end{equation}
is negligible with respect to $\psi$, i.e. $\psi(A)=0$. In particular, for $\psi$--almost every $x\in \Omega$, we have
\begin{equation}
\label{eq:Lebesgue points for Wgensup}
\lim\limits_{n\to
\infty}\left(\eta_n\fint_{B_{\e_n/2}(x)}\fint_{B_{\e_n/2}(x)} W(z,y)d\mu(z)d\mu(y)\right)=0.
\end{equation} 
\end{theorem}

\begin{proof}
{\it Proof of part 1:}
We prove inequality \eqref{eq:Besov constant is more that integral of lower limiting double integralgen}. Since $\mu$ is Ahlfors $s$-regular, there exist constants $c,C$ such that for every $x\in \Omega$ and for every $0<r<\infty$ with $B_r(x)\subset \Omega$, we have:
\begin{equation}
\label{eq:Ahlfors conditiongen}
cr^s\leq \mu(B_r(x))\leq Cr^s.
\end{equation}
Note that for every $x,y\in \Omega$ and $n\in\N$, we get from condition \eqref{eq:Ahlfors conditiongen}
\begin{equation}
\label{eq:use of Ahlfors conditiongen}
\frac{1}{\e_n^s}=\frac{1}{2^s(\e_n/2)^s}\geq \frac{c}{2^s\mu(B_{\e_n/2}(x))},\qquad
\frac{1}{\mu(B_{\e_n}(y))}\geq \frac{1}{C\e_n^s}\geq \frac{c}{C\,2^s\mu(B_{\e_n/2}(x))}.
\end{equation}
Observe that for every $x\in\Omega$, we have $B_{(3/2)\e_n}(x)\subset \Omega$ for all sufficiently large $n$. Hence, if $y\in B_{\e_n/2}(x)$, then $B_{\e_n}(y)\subset B_{(3/2)\e_n}(x)$,
and consequently $\chi_{\Omega}(z)=1$ for every $z\in B_{\e_n}(y)$.
Set 
\begin{equation}
\label{eq:def of Hn}
H_n(y):=\frac{1}{\e^s_n}\zeta\big(B_{\e_n}(y)\big)\Bigg\{\fint_{B_{\e_n}(y)}\chi_\Omega(z) W(z,y)d\mu(z)\Bigg\}.
\end{equation}
For every $x\in\Omega$, for every sufficiently large $n$ such that $B_{(3/2)\e_n}(x)\subset \Omega$ and $\e_n<\gamma'$, and using property \eqref{eq:property of zeta3}, we obtain
\begin{multline}
\label{eq:estimate from below for a density termgen}
\sup\limits_{0<\rho<
2\e_n}\frac{1}{\zeta\big(B_{\rho}(x)\big)}\int_{B_{\rho}(x)}H_n(y)d\mu(y)
\geq \frac{1}{\zeta\big(B_{\e_n}(x)\big)}\int_{B_{\e_n}(x)}H_n(y)d\mu(y)
\\
\geq \frac{1}{\zeta\big(B_{\e_n}(x)\big)}\int_{B_{\e_n/2}(x)}H_n(y)d\mu(y)
\geq \frac{1}{\lambda'\e^s_n}\int_{B_{\e_n/2}(x)}\Bigg\{\fint_{B_{\e_n}(y)}W(z,y)d\mu(z)\Bigg\}d\mu(y).
\end{multline}
Note that if $y \in B_{\varepsilon_n/2}(x)$, then $B_{\varepsilon_n/2}(x) \subset B_{\varepsilon_n}(y)$.
Therefore, we get from \eqref{eq:estimate from below for a density termgen} and \eqref{eq:use of Ahlfors conditiongen}
\begin{multline}
\label{eq:estimate from below for a density term1gen3}
\sup\limits_{0<\rho<
2\e_n}\frac{1}{\zeta\big(B_{\rho}(x)\big)}\int_{B_{\rho}(x)}H_n(y)d\mu(y)
\geq \frac{c}{\lambda' C\,2^s}
\frac{1}{\e^s_n}\int_{B_{\e_n/2}(x)}\Bigg\{\fint_{B_{\e_n/2}(x)}W(z,y)d\mu(z)\Bigg\}d\mu(y)
\\
\geq \frac{c^2}{\lambda'C\,2^{2s}}
\fint_{B_{\e_n/2}(x)}\Bigg\{\fint_{B_{\e_n/2}(x)}W(z,y)d\mu(z)\Bigg\}d\mu(y).
\end{multline}
Set $C':=C'(\zeta,\mu):=\frac{c^2}{\lambda'C\,2^{2s}}$.
Multiplying \eqref{eq:estimate from below for a density term1gen3} by $\eta_n$ and taking the lower limit as $n\to \infty$, we obtain
\begin{multline}
\label{eq:estimate from below for a density term1gen}
\liminf\limits_{n\to \infty}
\Bigg(\sup\limits_{0<\rho<
2\e_n}\frac{\eta_n}{\zeta\big(B_{\rho}(x)\big)}\int_{B_{\rho}(x)}H_n(y)d\mu(y)\Bigg)
\\
\geq C'\liminf\limits_{n\to \infty}\left(\eta_n
\fint_{B_{\e_n/2}(x)}\Bigg\{\fint_{B_{\e_n/2}(x)}W(z,y)d\mu(z)\Bigg\}d\mu(y)\right).
\end{multline}
Let us denote for each $n\in \N$ and Borel set $B\subset \Omega$
\begin{equation}
\label{mundefin}
\tilde{\mu}_n(B):=\int_{B}\eta_nH_n(y)d\mu(y),
\end{equation}
and for every set $E\subset X$, we define $\mu_n(E):=\inf\Set{\tilde{\mu}_n(B)}[B\,\,\text{is Borel},E\subset B]$. Note that $\tilde{\mu}_n(B)=\mu_n(B)$ for Borel sets $B$. In addition, for every $x\in X$ we denote
\begin{equation}
g(x):=C'\liminf\limits_{n\to \infty}\left(\eta_n\fint_{B_{\e_n/2}(x)}\Bigg\{\fint_{B_{\e_n/2}(x)}W(z,y)d\mu(z)\Bigg\}d\mu(y)\right).
\end{equation}
Thus, \eqref{eq:estimate from below for a density term1gen} reeds as
\begin{equation}
\liminf\limits_{n\to \infty}\Bigg(\sup\limits_{0<\rho<
2\e_n}\frac{\mu_n(B_{\rho}(x))}{\zeta(B_{\rho}(x))}\Bigg)\geq g(x),\quad \forall x\in \Omega.
\end{equation} 
Notice that $\mu_n$ is a Borel measure, meaning that it is an outer measure 
for which all Borel sets are measurable, and that, by Corollary \ref{cor:double average is Borel}, $g$ is $\psi$--measurable. Therefore, we get from Lemma \ref{lem:the upper density lemma for a sequence of Radon measures1}
\begin{equation}
\label{eq:generalized density theorem for non-negative functions8gen}
\lambda\left( \liminf_{n \to \infty} \mu_n(\Omega)\right) \geq \int_{\Omega}g(x) \, d\psi(x).
\end{equation}
It proves \eqref{eq:Besov constant is more that integral of lower limiting double integralgen} with $C(\zeta,\mu):=\frac{1}{\lambda }\frac{c^2}{\lambda'C\,2^{2s}}$. 
\\
{\it Proof of part 2:}
For the second part of the theorem, let us denote for every $l\in\N$
\begin{equation}
D_l:=\Set{x\in \Omega}[\liminf\limits_{n\to
\infty}\left(\eta_n\fint_{B_{\e_n/2}(x)}\fint_{B_{\e_n/2}(x)} W(z,y)d\mu(z)d\mu(y)\right)>1/l].
\end{equation}
Note that $D=\bigcup_{l=1}^\infty D_l$, and from \eqref{eq:Besov constant is more that integral of lower limiting double integralgen} and the finiteness assumption of its left part, we get by Chebyshev inequality
\begin{multline}
\psi(D_l)\leq l \int_{\Omega}\liminf\limits_{n\to
\infty}\left(\eta_n\fint_{B_{\e_n/2}(x)}\fint_{B_{\e_n/2}(x)} W(z,y)d\mu(z)d\mu(y)\right)d\psi(x)
\\
\leq l\,C(\zeta,\mu)^{-1}\liminf\limits_{n\to\infty}\left(\eta_n\int_\Omega H_n(x)d\mu(x)\right)
<\infty.
\end{multline}
It proves that $D$ is $\sigma$-finite with respect to $\psi$. We get \eqref{eq:Lebesgue points for Wgen} from the definition of the set $D$.
\\
{\it Proof of part 3:}
For the third part of the theorem, we again choose $\mu_n$ as in \eqref{mundefin}. Then by \eqref{eq:Besov constant is more that integral of lower limiting double integralgensum} and \eqref{mundefin} we have
\begin{equation}
\label{eq:Besov constant is more that integral of lower limiting double integralgensummu}
\sum\limits_{n=1}^{\infty}\mu_n(\Omega)
<\infty\,.
\end{equation}
By \eqref{eq:estimate from below for a density term1gen3} we have
for $2\e_n<\delta$
\begin{multline}
\label{eq:estimate from below for a density term1genjjkkjj}
\sup_{0 < \rho < \delta}
\frac{ \mu_n\big(B_\rho(x)\big)}{ \zeta\big(B_\rho(x)\big)}
\geq \sup\limits_{0<\rho<
2\e_n}\frac{1}{\zeta\big(B_{\rho}(x)\big)}\left(\int_{B_{\rho}(x)}\eta_nH_n(y)d\mu(y)\right)
\\
\geq C'\left(\eta_n
\fint_{B_{\e_n/2}(x)}\Bigg\{\fint_{B_{\e_n/2}(x)}W(z,y)d\mu(z)\Bigg\}d\mu(y)\right).
\end{multline}
In particular,
\begin{multline}
\label{eq:estimate from below for a density term1genjjkkjjth}
\lim_{\delta\to 0^+}\limsup_{n \to \infty}\left( \sup_{0 < \rho < \delta}
\frac{ \mu_n\big(B_\rho(x)\big)}{ \zeta\big(B_\rho(x)\big)} \right)
\\
\geq C'\limsup_{n \to \infty}\left(\eta_n
\fint_{B_{\e_n/2}(x)}\Bigg\{\fint_{B_{\e_n/2}(x)}W(z,y)d\mu(z)\Bigg\}d\mu(y)\right).
\end{multline}
From \eqref{eq:estimate from below for a density term1genjjkkjjth} we see that the assumption \eqref{eq:bound from below on the generalized upper density:limsupversion} of Lemma \ref{lem:the upper density lemma for a sequence of Radon measures:limsupversion} is fulfilled for $x\in A$, where $A$ is defined in \eqref{eq:def of A}, and we also have \eqref{eq:Besov constant is more that integral of lower limiting double integralgensummu}. Therefore, $\psi(A)=0$.
\end{proof}

\section{The set of points that are not general average Lebesgue points of functions in $B^r_{q,\infty}$ is $\sigma$-finite with respect to the Hausdorff measure $\mathcal{H}^{s-rq}$}

A particularly important case of Carathéodory's construction is the Hausdorff measure, which we define separately below:
\begin{definition}[Definition of the Hausdorff measure]
\label{def:Hausdorff measures}
Let $X$ be a metric space, $s\in [0,\infty)$, and $\delta\in (0,\infty]$. The $s$-dimensional approximating Hausdorff measure of size $\delta$ denoted as $\mathcal{H}^s_\delta$ is defined for every $E\subset X$ by
\begin{equation}
\label{eq:def of Haus. mea.}
\mathcal{H}^s_{\delta}(E):= \inf\Set{\sum_{i=1}^\infty \left(\diam(E_i)\right)^s}[E \subset \bigcup_{i=1}^\infty E_i,\, \diam(E_i) \leq \delta ].
\end{equation}
The $s$-dimensional Hausdorff measure is defined to be $\mathcal{H}^s(E)=\lim_{\delta\to 0^+}\mathcal{H}^s_{\delta}(E)$.
\end{definition}
As mentioned in the discussion of Carathéodory's construction, the resulting measure is Borel regular if the family $\mathcal{F}$ consists of Borel sets. The value of the measure in \eqref{eq:def of Haus. mea.} does not change if we restrict the sets $E_i$ to be closed or Borel. Indeed, the diameter of a set coincides with the diameter of its closure. Therefore, $\mathcal{H}^s$ is a Borel regular measure.

We now restate and prove Theorem~\ref{thm:sigma finiteness of limiting average with respect to Hausdorff measure,intro}.
\begin{theorem}
\label{cor:sigma finiteness of limiting average with respect to Hausdorff measure}
Let $X$ be a metric space, and let $\mu$ be an Ahlfors $s$-regular measure on $X$ for some $s\in (0,\infty)$. Let $(Y,\rho)$ be a metric space. Let $0\leq r\leq 1$, $0<q<\infty$ be such that $rq\leq s$, and let $u\in B^r_{q,\infty}(X,\mu;Y)$. 
Then, for any sequence $\e_n>0$ such that
$\lim_{n\to\infty}\e_n=0$ there exists a $\mathcal{H}^{s-rq}$
$\sigma$-finite set $D\subset X$, such that $\forall x\in X\setminus D$
\begin{equation}
\label{eq:fine property of BV^q funcions}
\liminf_{n\to \infty}\fint_{B_{\e_n}(x)}\fint_{B_{\e_n}(x)}\rho\left(u(z),u(y)\right)^qd\mu(y)d\mu(z)=0.
\end{equation}
\end{theorem}

\begin{proof}
We intend to use Theorem \ref{thm:Besov constant controls double average integral} with $\Omega=X$, $\eta_n=1$, $\zeta(A):=\left(\text{diam}(A)\right)^{s-rq}$, and $W(z,y):=\rho\left(u(z),u(y)\right)^q$. Since $u\in B^r_{q,\infty}(X,\mu;Y)$, then 
\begin{equation}
\liminf\limits_{n\to\infty}\int_X\Bigg\{\fint_{B_{(2\e_n)}(x)}\frac{\rho\left(u(z),u(x)\right)^q}{(2\e_n)^{rq}}d\mu(z)\Bigg\}d\mu(x)\leq \|u\|^q_{B^{r}_{q,\infty}(X,\mu;Y)}<\infty.
\end{equation} 
From the first and second parts of Theorem \ref{thm:Besov constant controls double average integral} we get for the sequence $2\e_n$ an $\mathcal{H}^{s-rq}$ $\sigma$-finite set $D\subset X$ such that for every $x\in X\setminus D$ we get \eqref{eq:fine property of BV^q funcions}.
\end{proof}

\section{The exceptional set of non-average Lebesgue points of $W^{r,q}$-functions has zero $(s-rq)$-dimensional Hausdorff measure}

The following lemma will help us determine the Borel measurability of a function that will appear in Theorem~\ref{thm:hausdorff-dim-D}.
\begin{lemma}
\label{lem:contability of sup of Borel functions}
Let $X$ be a metric space. Let $\varphi:\mathcal{P}(X)\to [0,\infty]$ be a set function such that whenever $A\subset B\subset X$, one has $\varphi(A)\leq \varphi(B)$. Let $h:(0,\infty)\to (0,\infty)$ be a continuous function. Fix $\sigma_0\in (0,\infty]$. Then
\begin{equation}
\sup_{0 < \sigma < \sigma_0}
\frac{\varphi\big(B_\sigma(x)\big)}{h(\sigma)}
=
\sup_{0 < \sigma < \sigma_0,\ \sigma\in\mathbb{Q}}
\frac{\varphi\big(B_\sigma(x)\big)}{h(\sigma)}.
\end{equation}
In particular, if $\varphi$ is a Borel measure and finite on bounded sets, then the function
\begin{equation}
Q(x):=\sup_{0 < \sigma < \sigma_0}
\frac{\varphi\big(B_\sigma(x)\big)}{h(\sigma)}
\end{equation}
is Borel measurable.
\end{lemma}

\begin{proof}
Let $\sigma\in (0,\sigma_0)$. For every $n\in\mathbb{N}$, choose $\sigma_n\in (0,\sigma_0)\cap \mathbb{Q}$ such that $\sigma_n\downarrow \sigma$. Then, using monotonicity of $\varphi$ and that $\sigma_n>\sigma$, we obtain
\begin{equation}
\label{eq:estimate for phi divided by h}
\frac{\varphi\big(B_\sigma(x)\big)}{h(\sigma)}
\leq \frac{\varphi\big(B_{\sigma_n}(x)\big)}{h(\sigma_n)}\,
\frac{h(\sigma_n)}{h(\sigma)}
\leq \left(\sup_{0 < q < \sigma_0,\, q\in\mathbb{Q}}
\frac{\varphi\big(B_q(x)\big)}{h(q)}\right)
\frac{h(\sigma_n)}{h(\sigma)}.
\end{equation}
Since \eqref{eq:estimate for phi divided by h} holds for every $n\in\mathbb{N}$ and $h$ is continuous at $\sigma$, taking the limit as $n\to\infty$ yields
\begin{equation}
\label{eq:estimate for phi divided by h1}
\frac{\varphi\big(B_\sigma(x)\big)}{h(\sigma)}
\leq \sup_{0 < q < \sigma_0,\ q\in\mathbb{Q}}
\frac{\varphi\big(B_q(x)\big)}{h(q)}.
\end{equation}
Since $\sigma\in (0,\sigma_0)$ is arbitrary, we obtain
\begin{equation}
\sup_{0 < \sigma < \sigma_0}
\frac{\varphi\big(B_\sigma(x)\big)}{h(\sigma)}
\leq \sup_{0 < q < \sigma_0,\ q\in\mathbb{Q}}
\frac{\varphi\big(B_q(x)\big)}{h(q)}.
\end{equation}
The reverse inequality is immediate. Assume now that $\varphi$ is a Borel measure and finite on bounded sets. Then, using Lemma \ref{lem:uper semicontinuity of h}, we get for each fixed $\sigma>0$, that the function
$x\mapsto \varphi(B_\sigma(x))$ is upper semicontinuous, and hence Borel measurable. Thus, the function $Q$ is a countable supremum of Borel functions and is therefore itself Borel.
\end{proof}

We now restate and prove Theorem~\ref{thm:hausdorff-dim-D,intro}.
\begin{theorem}
\label{thm:hausdorff-dim-D}
Let $(X,d)$ be a metric space, and let $\mu$ be an Ahlfors $s$-regular measure on $X$ for some $s\in (0,\infty)$. Let $(Y,\rho)$ be a metric space. Let $0<r<1$, $0<q<\infty$ be such that $rq\leq s$, and let $u\in W^{r,q}_{\text{loc}}(X,\mu;Y)$. 
Define
\begin{equation}
\label{eq:def-D1}
D := \Set{x \in X}[
\limsup_{\varepsilon \to 0^+} \fint_{B_{\varepsilon}(x)} \fint_{B_{\varepsilon}(x)}
\rho\left(u(z),u(y)\right)^q\, d\mu(z)\, d\mu(y) > 0].
\end{equation}
Then $\mathcal{H}^{s - rq}(D)=0$. In other words, $\mathcal{H}^{s - rq}$–almost every point is an average Lebesgue point in the sense of Definition~\ref{def:lower average Lebesgue point}.
\end{theorem}

\begin{proof}
The idea of the proof is as follows. We apply Lemma~\ref{lem:the upper density lemma for a sequence of Radon measures1} together with the assumption $u\in W^{r,q}_{\mathrm{loc}}(X,\mu;Y)$ to construct a function $g$ such that $g(x)=0$ for $\mathcal{H}^{s-rq}$–almost every $x$. We then show that the upper limit in $D$ is bounded above by $g$ at $\mathcal{H}^{s-rq}$–almost every $x\in X$.

Set $V(z,y):=\rho\left(u(z),u(y)\right)^q$. Define a Borel measure $\mu_n$ on $X$ as follows. If $B\subset X$ is a Borel set, set
\begin{equation}
\label{eq:F(B) is positivehtyjjj}
\tilde{\mu}_n(B)
:=
\int_{B} \int_{B_{1/n}(y)}
\frac{V(z,y)}{d(z,y)^{s+rq}}\, d\mu(z)\, d\mu(y).
\end{equation}
For an arbitrary set $A\subset X$, define
\[
\mu_n(A)
:=
\inf\Set{\tilde{\mu}_n(B)}[B \text{ is Borel},\ A\subset B].
\]
Note that $\tilde{\mu}_n(B)=\mu_n(B)$ for Borel sets $B$. Observe that, since $\mu$ is Ahlfors regular, every singleton $\{y_0\}$ has measure zero. Thus, since  $u\in W^{r,q}_{\text{loc}}(X,\mu;Y)$, for every ball $B\subset X$ we have by dominated convergence theorem
\begin{equation}
\label{eq:limit of mu n in ball is zero}
\lim\limits_{n\to\infty}\mu_n(B)=\int_{B}\left(\lim\limits_{n\to\infty} \int_{B_{\frac{1}{n}}(y)}
\frac{V(z,y)}{d(z,y)^{s+rq}}d\mu(z)\right)d\mu(y)=\int_{B}\int_{\{y\}}
\frac{V(z,y)}{d(z,y)^{s+rq}}d\mu(z)d\mu(y)=0.
\end{equation}
Next, given $\alpha\geq 0$, and an arbitrary positive sequence $r_n\downarrow 0$, define
\begin{equation}
\label{eq:bound from below on the generalized upper density(1)hj}
g(x):=\liminf_{n \to \infty} \left( \sup_{0 < \sigma < r_n}
\frac{\mu_n\big(B_\sigma(x)\big)}{2^\alpha\sigma^\alpha}  \right). 
\end{equation}
Note that the limit in~\eqref{eq:bound from below on the generalized upper density(1)hj} as $n\to\infty$ exists; however, we do not use this fact.
Then, obviously,
\begin{equation}
\label{eq:bound from below on the generalized upper density(1)hjjjkj}
\liminf_{n \to \infty} \left( \sup_{0 < \sigma < r_n}
\frac{\mu_n\big(B_\sigma(x)\big)}{2^\alpha\sigma^\alpha}\right)\geq g(x)\quad \forall x\in X.
\end{equation}
By Lemma~\ref{lem:contability of sup of Borel functions}, the function $g$ is Borel measurable, and therefore $\mathcal{H}^\alpha$-measurable. Thus, using \eqref{eq:limit of mu n in ball is zero}, and applying Lemma \ref{lem:the upper density lemma for a sequence of Radon measures1} with $\zeta(A)=(\text{diam}(A))^\alpha$, we obtain for every open ball $B_0\subset X$
\begin{equation}
\label{eq:applying generalized density theorem}
0=\lambda \left( \liminf_{n \to \infty} \mu_n(B_0)\right) \geq \int_{B_0}g(x) \, d\mathcal{H}^\alpha(x).
\end{equation}
Thus, by \er{eq:F(B) is positivehtyjjj} and \eqref{eq:applying generalized density theorem}, we deduce
that for every $\alpha\geq 0$, for every positive sequence $r_n\downarrow 0$, and for every open ball $B_0\subset X$ we have
\begin{equation}
\label{eq:bound from below on the generalized upper density(1)hjjkj}
g(x)=0\quad\text{for $\mathcal{H}^\alpha$-almost every }x\in B_0.
\end{equation}
In particular, for the choice $r_n=1/n$, for every $\alpha\geq 0$, for every positive sequence $r_n\downarrow 0$, and for every open ball $B_0\subset X$, we obtain from~\eqref{eq:bound from below on the generalized upper density(1)hjjkj} that, for $\mathcal{H}^\alpha$–almost every $x\in B_0$,
\begin{multline}
\label{eq:bound from below on the generalized upper density(1)hjjkjjkfb}
\limsup_{\sigma\to 0^+} \left(
\frac{1}{\sigma^{s+rq+\alpha}}\int_{B_\sigma(x)} \int_{B_{\sigma}(y)}
V(z,y)\, d\mu(z)\, d\mu(y)  \right)
\\
\leq\limsup_{\sigma\to 0^+} \left(
\frac{1}{\sigma^\alpha}\int_{B_\sigma(x)} \int_{B_{\sigma}(y)}
\frac{V(z,y)}{d(z,y)^{s+rq}}\, d\mu(z)\, d\mu(y)  \right)\quad \Big[\text{since $d(z,y)\leq \sigma$ for $z\in B_\sigma(y)$}\Big]
\\
=\lim_{n \to \infty} \left( \sup_{0 < \sigma < 1/n}
\frac{1}{\sigma^\alpha}\int_{B_\sigma(x)} \int_{B_{\sigma}(y)}
\frac{V(z,y)}{d(z,y)^{s+rq}}\, d\mu(z)\, d\mu(y)  \right)
\\
\leq\lim_{n \to \infty} \left( \sup_{0 < \sigma < 1/n}
\frac{1}{\sigma^\alpha}\int_{B_\sigma(x)} \int_{B_{1/n}(y)}
\frac{V(z,y)}{d(z,y)^{s+rq}}\, d\mu(z)\, d\mu(y)  \right)=2^\alpha g(x)=0.
\end{multline}
Therefore, by \eqref{eq:bound from below on the generalized upper density(1)hjjkjjkfb}, we obtain
\begin{multline}
\label{eq:bound from below on the generalized upper density(1)hjjkjjkfbty}
\limsup_{\sigma\to 0^+} \left(
\frac{1}{\sigma^{s+rq+\alpha}}\int_{B_{\sigma}(x)} \int_{B_{\sigma}(x)}
V(z,y)\, d\mu(z)\, d\mu(y)  \right)
\\
\leq\limsup_{\sigma\to 0^+} \left(
\frac{1}{\sigma^{s+rq+\alpha}}\int_{B_{\sigma}(x)} \int_{B_{2\sigma}(y)}
V(z,y)\, d\mu(z)\, d\mu(y)  \right)
\\
\leq \limsup_{\sigma\to 0^+} \left(
\frac{1}{\sigma^{s+rq+\alpha}}\int_{B_{2\sigma}(x)} \int_{B_{2\sigma}(y)}
V(z,y)\, d\mu(z)\, d\mu(y)  \right)=0\quad\text{for $\mathcal{H}^\alpha$-almost every }x\in B_0.
\end{multline}
Then, since $B_0\subset X$ is an arbitrary open ball, for every $\alpha\geq 0$ we deduce
\begin{equation}
\label{eq:bound from below on the generalized upper density(1)hjjkjjkfbtyyh}
\limsup_{\sigma\to 0^+} \left(
\frac{1}{\sigma^{s+rq+\alpha}}\int_{B_{\sigma}(x)} \int_{B_{\sigma}(x)}
V(z,y)\, d\mu(z)\, d\mu(y)  \right)=0\quad\text{for $\mathcal{H}^\alpha$-almost every }x\in X.
\end{equation}
In particular, the result follows by taking $\alpha=(s-rq)$ in \eqref{eq:bound from below on the generalized upper density(1)hjjkjjkfbtyyh} and using that $\mu\left(B_\sigma(x)\right)\ge c\sigma^s$.
\end{proof}

\section{The logarithmic Hausdorff measure}
In this section, we introduce an auxiliary measure, which we call the logarithmic Hausdorff measure. We will use it to prove Lemma~\ref{lem:varphi-Hasdorff measure of D}, which is the main ingredient in establishing an upper bound on the Hausdorff dimension of non-Lebesgue points of Besov functions, as stated in Theorem~\ref{thm:varphi-Hasdorff measure of D}.

The following definition is a special case of Carathéodory’s construction. It gives rise to a measure that we call the \emph{logarithmic Hausdorff measure}, which we denote by the Greek letter $\Lambda$. More precisely,
\begin{definition}[Definition of the logarithmic Hausdorff measure]
\label{def:ln-Hausdorff measures}
Let $X$ be a metric space and $s\in [0,\infty)$, $\delta\in (0,\infty]$, $\beta\in (0,\infty)$. The $s$-dimensional approximating logarithmic Hausdorff measure of size $\delta$ and decreasing factor $\beta$ denoted as $\Lambda^{s,\beta}_{\delta}$ is defined for every $E\subset X$ by
\begin{equation}
\label{eq:def of ln-Haus. mea.}
\Lambda^{s,\beta}_{\delta}(E):= \inf\Set{\sum_{i=1}^\infty g_{s,\beta}(\diam E_i)}[E \subset \bigcup_{i=1}^\infty E_i,\, \diam(E_i) \leq \delta ],
\end{equation}
where 
\begin{equation}\label{ggg}
g_{s,\beta}:[0,\infty]\to [0,\infty],\quad g_{s,\beta}(r):=
\begin{cases}
\frac{r^s}{|\ln r|^{\beta}}\quad & if\quad  r\in (0,\infty)\setminus\{1\}
\\
0\quad & if\quad  r=0
\\
\infty \quad & if\quad  r\in \{1,\infty\}
\end{cases}.
\end{equation}
The $s$-dimensional logarithmic Hausdorff measure with decreasing factor $\beta$ is defined by
\[
\Lambda^{s,\beta}(E)=\lim_{\delta\to 0^+}\Lambda^{s,\beta}_{\delta}(E).
\]
\end{definition}
We refer to $\beta$ as a \emph{decreasing factor}, since, for every small enough $\delta$, the larger $\beta$ is, the smaller the function $g_{s,\beta}$ becomes, and hence the smaller the measure $\Lambda^{s,\beta}_{\delta}$.

\begin{proposition}
\label{prop:properties of the ln-measure}
Let $X$ be a metric space, $s\in [0,\infty)$, and $\beta\in(0,\infty)$. Let $\Lambda^{s,\beta}$ be the logarithmic measure on $X$. Then the measure $\Lambda^{s,\beta}$ is obtained via Carathéodory's construction from $\zeta(A):=g_{s,\beta}\left(\diam A\right)$, $A\subset X$, where $g_{s,\beta}$ is defined in~\eqref{ggg}. In addition, the function $\zeta$ satisfies the four conditions
\eqref{eq:property of zeta1}, \eqref{eq:property of zeta}, \eqref{eq:property of zeta3}, and \eqref{eq:property of zeta5}.
\end{proposition}
\begin{proof}
Notice that, by definition, $\Lambda^{s,\beta}$ is indeed obtained from Carathéodory's construction associated with $\zeta$, when $\zeta$ is defined on all subsets of $X$. For condition~\eqref{eq:property of zeta1}, note that for every $0<\sigma<1/2$,
\begin{equation}
\zeta\big(B_{\sigma}(x)\big)=g_{s,\beta}(2\sigma)=\frac{(2\sigma)^{s}}{|\ln(2\sigma)|^{\beta}}
\quad \Longrightarrow\quad
0<\zeta\big(B_{\sigma}(x)\big)<\infty.
\end{equation}
We now verify condition~\eqref{eq:property of zeta}. There exists a sufficiently small $0<\gamma<1/10$ such that for every $0<\sigma<\gamma$ we have
\begin{multline}
\frac{\zeta\big(B_{5\sigma}(x)\big)}{\zeta\big(B_{\sigma}(x)\big)}
=g_{s,\beta}(10\sigma)\frac{1}{g_{s,\beta}(2\sigma)}
=
\frac{(10\sigma)^{s}}{|\ln(10\sigma)|^{\beta}}
\frac{|\ln (2\sigma)|^{\beta}}{(2\sigma)^{s}}
\\
=\frac{(10\sigma)^{s}}{|\ln\sigma +\ln 10|^{\beta}}
\frac{|\ln \sigma+\ln 2|^{\beta}}{(2\sigma)^{s}}
=
5^{s}
\left(\frac{\left|-1+\frac{\ln 2}{|\ln \sigma|}\right|}{\left|-1+\frac{\ln 10}{|\ln \sigma|}\right|}\right)^\beta
\leq 5^{s}2^\beta.
\end{multline}
Therefore, condition~\eqref{eq:property of zeta} holds with $\lambda=5^{s}2^\beta$. 
Condition~\eqref{eq:property of zeta3} follows since the value of $\zeta\big(B_{\sigma}(x)\big)$ does not depend on the center point $x$, and hence
\begin{equation}
\zeta\big(B_{\sigma}(x)\big)=g_{s,\beta}(2\sigma)=\zeta\big(B_{\sigma}(y)\big) \quad \forall y \in X.
\end{equation}
Condition~\eqref{eq:property of zeta5} holds trivially because the function $x\mapsto \zeta\big(B_{\sigma}(x)\big)$ is constant, and, in particular, a Borel function.
\end{proof}

\begin{proposition}
Let $X$ be a metric space, $s\in [0,\infty)$, and $\beta\in (0,\infty)$. For every $0<\delta\le e^{-1}$ and every $A\subset X$,
\begin{equation}
\label{eq:delta logarithmic measure in A is controlled by delta Hausdorff measure}
\Lambda^{s,\beta}_{\delta}(A)\le \mathcal{H}_{\delta}^{s}(A).
\end{equation}
In particular,
\begin{equation}
\label{eq:logarithmic measure in A is controlled by Hausdorff measure}
\Lambda^{s,\beta}(A)\le \mathcal{H}^{s}(A).
\end{equation}
\end{proposition}

\begin{proof}
For every $\beta>0$ and every $0<\delta\le e^{-1}$, we have $|\ln(\diam(E))|\ge 1$ for every set $E$ satisfying
$0<\diam(E)\le \delta$. Therefore,
\[
(\diam(E))^{s}|\ln(\diam(E))|^{-\beta}\le (\diam(E))^{s}.
\]
Thus, $g_{s,\beta}(\diam(E))\le (\diam(E))^{s}$ for every set $E$ with
$\diam(E)\le \delta$ (including the case $\diam(E)=0$). Using this inequality, for every $\delta$–cover $\{E_i\}$ of a set $A$, we obtain
\[
\sum_i(\diam(E_i))^{s}|\ln(\diam(E_i))|^{-\beta}
\le 
\sum_i(\diam(E_i))^{s}.
\]
Taking the infimum over all such covers yields \eqref{eq:delta logarithmic measure in A is controlled by delta Hausdorff measure}. Passing to the limit as $\delta\to 0^+$ in \eqref{eq:delta logarithmic measure in A is controlled by delta Hausdorff measure}, we obtain \eqref{eq:logarithmic measure in A is controlled by Hausdorff measure}.
\end{proof}

\begin{definition}[Hausdorff dimension and logarithmic Hausdorff dimension]
\label{def:Hausdorff dimension and logarithmic Hausdorff dimension}
Let $X$ be a metric space.  
For a set $E \subset X$ and $\beta\in (0,\infty)$, define the Hausdorff dimension of $E$ and the logarithmic Hausdorff dimension of $E$, respectively, by 
\begin{equation}
\label{eq:def-Log-dim}
\begin{cases}
\mathcal{H}_{\dim} E 
:= \inf\Set{s\ge 0}[\mathcal H^s(E)=0]
= \sup\Set{s\ge 0}[\mathcal H^s(E)=\infty],
\\
\Lambda_{\dim} E 
:= \inf\Set{s\ge 0}[\Lambda^{s,\beta} (E)=0]
= \sup\Set{s\ge 0}[\Lambda^{s,\beta}(E)=\infty].
\end{cases}
\end{equation}
\end{definition}

Note that we omit $\beta$ in the notation $\Lambda_{\dim} E$. The justification is given in the following proposition, which states that the logarithmic Hausdorff measure and the usual Hausdorff measure determine the same dimension:
\begin{proposition}
\label{prop:equality between Hausdorff-D}
Let $X$ be a metric space, let $\beta\in (0,\infty)$, $s\in [0,\infty)$, and let $E\subset X$.  
Then the Hausdorff dimension defined via $\mathcal{H}^s$ coincides with the dimension defined via the logarithmic Hausdorff measure $\Lambda^{s,\beta}$:
\begin{equation}
\label{eq:equality between dimensions}
\mathcal{H}_{\dim} E = \Lambda_{\dim} E.
\end{equation}
In particular, the logarithmic dimension does not depend on $\beta$.
\end{proposition}

\begin{proof}
Fix $\beta\in (0,\infty)$. Let $s>0$ and $\e>0$ be arbitrary. There exists $r_0=r_0(\varepsilon,\beta)$ such that
for all $0<r<r_0$,
\begin{equation}
\label{eq:gauge-comparison}
r^{s+\varepsilon}
\;\le\;
\frac{r^s}{|\ln r|^\beta}
\;\le\;
r^{s-\varepsilon}.
\end{equation}
Inequalities \eqref{eq:gauge-comparison} follow from $
\lim_{r\to 0^+} r^\e |\ln r|^\beta = 0,\, \lim_{r\to 0^+} r^{-\e} |\ln r|^{\beta} = \infty.$
Let $\delta\in(0,r_0)$ and let $\{E_i\}$ be a $\delta$--cover of $E$.
Using \eqref{eq:gauge-comparison}, we obtain
\[
\sum_i \diam(E_i)^{s+\varepsilon}
\;\le\;
\sum_i g_{s,\beta}(\diam(E_i))\;\le\;
\sum_i \diam(E_i)^{s-\varepsilon}.
\]
Taking the infimum over all such covers and then letting $\delta\to 0^+$ yields
\begin{equation}
\label{eq:measure-comparison}
\mathcal H^{s+\e}(E)
\;\le\;
\Lambda^{s,\beta}(E)
\;\le\;
\mathcal H^{s-\e}(E).
\end{equation}
Let $d:=\mathcal{H}_{\dim} E$. If $s>d$, choose $\e>0$ such that $s-\e>d$. Then $\mathcal H^{s-\e}(E)=0$, and by \eqref{eq:measure-comparison} we obtain $\Lambda^{s,\beta}(E)=0$. Hence $\Lambda_{\dim}E \le s$. Since $s$ is an arbitrary number greater than $d$, we get $\Lambda_{\dim} E \le d$. If $s<d$, choose $\e>0$ such that $s+\e<d$. Then $\mathcal H^{s+\e}(E)=\infty$, and again by \eqref{eq:measure-comparison} we have $\Lambda^{s,\beta}(E)=\infty$. Hence $\Lambda_{\dim}E \ge s$. Since $s$ is an arbitrary number less than $d$, we get $\Lambda_{\dim} E \ge d$. Combining the two inequalities yields \eqref{eq:equality between dimensions}.
\end{proof}

\section{Hausdorff dimension of the set of Lebesgue points for $B^r_{q,\infty}$-functions}
\begin{lemma}
\label{lem:taking almost subsequence}
Let $\{A_n\}_{n\in\N}\subset \R$ be such that $\lim_{n\to \infty}A_n=0$. Let $\{\ell_n\}_{n\in\N}\subset \N$ be any sequence such that $\lim_{n\to \infty}\ell_n=\infty$. Then,
\begin{equation}
\lim_{n\to \infty}A_{\ell_n}=0.
\end{equation}
\end{lemma}
The point of the lemma is that $A_{\ell_n}$ is not necessarily a subsequence of $A_n$, since we do not assume that $\ell_n<\ell_{n+1}$ for every $n\in\N$.
\begin{proof}
Let $\e>0$. Then there exists $N_0$ such that $|A_n|<\e$ for all $n>N_0$. Since $\lim_{n\to \infty}\ell_n=\infty$, there exists $L$ such that $\ell_n>N_0$ for all $n>L$. Thus, for arbitrary $\e>0$, there exists $L$ such that for every $n>L$ we have $|A_{\ell_n}|<\e$. Therefore, $\lim_{n\to \infty}A_{\ell_n}=0$.
\end{proof}

\begin{lemma}
\label{lem:D equals Dseq}
Let $X$ be a metric space. Let $\Sigma$ be the collection of all closed balls in $X$. Let $F:\Sigma\to [0,\infty]$ be a monotone increasing set function, meaning that $A\subset B$, $A,B\in \Sigma$, implies $F(A)\leq F(B)$. Fix $\xi\in (0,\infty)$ and $\alpha\in [0,\infty)$. Define 
\begin{equation}
\label{def: def sequential D}
D_{seq}:=\Set{x\in X}[\limsup_{n\to \infty} \left(\frac{n^\alpha}{\left(e^{-n}\right)^{\xi}}F\left(B_{e^{-n}}(x)\right)\right)>0]\footnote{"seq" is for sequence},
\end{equation}
and
\begin{equation}
\label{def: def non-sequential D}
D:=\Set{x\in X}[\limsup_{\sigma\to 0^+} \left(\frac{|\ln \sigma|^\alpha}{\sigma^{\xi}}F\left(B_{\sigma}(x)\right)\right)>0].
\end{equation}
Then,
\begin{equation}
D=D_{seq}.
\end{equation}
\end{lemma}  

\begin{proof}
Note that $n^\alpha=|\ln e^{-n}|^\alpha$. Therefore, by the definition of the upper limit, we have $D_{seq}\subset D$. We now show that $X\setminus D_{seq}\subset X\setminus D$, and therefore $D\subset D_{seq}$. Let $x\in X\setminus D_{seq}$. We need to prove that for every sequence $\sigma_n$ of positive numbers converging to zero, we have
\begin{equation}
\label{eq:limit with general sigma n}
\lim_{n\to \infty} \left(\frac{|\ln \sigma_n|^\alpha}{\sigma_n^{\xi}}F\left(B_{\sigma_n}(x)\right)\right)=0.
\end{equation}
Without loss of generality assume that $\sigma_n\leq e^{-1}$ for every $n\in\N$. For every $n\in\N$ choose $\ell_n\in\N$ such that
\begin{equation}
\label{eq:choosing subsequence}
e^{-(\ell_n+1)}<\sigma_{n}\leq e^{-\ell_n}.  
\end{equation}
Thus, using \eqref{eq:choosing subsequence} and the monotonicity of $F$, we obtain
\begin{multline}
\label{eq:estimate for general sigma n}
\left|\ln\sigma_{n}\right|^\alpha\frac{1}{\sigma_{n}^\xi}F\left(B_{\sigma_{n}}(x)\right)
\leq |\ell_n+1|^\alpha\left(e^{(\ell_n+1)}\right)^{\xi}F\left(B_{e^{-\ell_n}}(x)\right)
\\
=e^{\xi}\left|1+\frac{1}{\ell_n}\right|^\alpha \left(\frac{\ell_n^{\alpha}}{\left(e^{-\ell_n}\right)^{\xi}}F\left(B_{e^{-\ell_n}}(x)\right)\right).
\end{multline}
Note that by \eqref{eq:choosing subsequence} and since $\sigma_n$ tends to zero when $n\to \infty$, we get $\lim_{n\to \infty}\ell_n=\infty$, and therefore, by the assumption that $x\notin D_{seq}$ and Lemma \ref{lem:taking almost subsequence}, we obtain
\begin{equation}
\label{eq:convergence with subsequence}
\lim_{n\to \infty}\left(\frac{\ell_n^{\alpha}}{\left(e^{-\ell_n}\right)^{\xi}}F\left(B_{e^{-\ell_n}}(x)\right)\right)=0.
\end{equation}
Therefore, taking the limit as $n\to \infty$ in \eqref{eq:estimate for general sigma n} and using \eqref{eq:convergence with subsequence} and that $\lim_{n\to \infty}\ell_n=\infty$, we obtain \eqref{eq:limit with general sigma n}. We conclude that $x\notin D$.
\end{proof}

\begin{lemma}
\label{lem:varphi-Hasdorff measure of D}
Let $X$ be a metric space, and let $\mu$ be an Ahlfors $s$-regular measure on $X$ for some $s\in (0,\infty)$. Let $0\leq r\leq 1$, $0<q<\infty$ be such that $rq\leq s$. Let $(Y,\rho)$ be a metric space, and let $u\in B^r_{q,\infty}(X,\mu;Y)$. Define for $\alpha \in [0,\infty)$ the set
\begin{equation}
\label{def: def D alpha}
D_\alpha:=\Set{x\in X}[\limsup_{\e \to 0^+} \left(|\ln\e|^\alpha\fint_{B_{\e}(x)} \fint_{B_{\e}(x)}
\rho\left(u(z),u(y)\right)^q\, d\mu(z)\, d\mu(y)\right)>0].
\end{equation}
Then, 
\begin{equation}
\label{def: def D alpha with en}
D_\alpha=\Set{x\in X}[\limsup_{n \to \infty} n^\alpha\fint_{B_{\e_n}(x)} \fint_{B_{\e_n}(x)}
\rho\left(u(z),u(y)\right)^q\, d\mu(z)\, d\mu(y)>0],\quad \e_n:=e^{-n},
\end{equation}
and if $\beta>1+\alpha$, then
\begin{equation}
\label{eq: D alpha is negligible w.r.t logarithmic measure}
\Lambda^{s-rq,\beta}(D_\alpha)=0.
\end{equation} 
In particular, the Hausdorff dimension of $D_\alpha$ is at most $s - rq$.
\end{lemma}

\begin{proof}
We intend to use the third part of Theorem~\ref{thm:Besov constant controls double average integral}.
Let us set
\begin{equation}
\label{eq:data for using theorem}
W(z,y):=\rho\big(u(z),u(y)\big)^q,\quad
\eta_n:=n^\alpha,\quad
\varepsilon_n:=e^{-n},\quad \zeta(A):=g_{s-rq,\beta}\!\left(\diam(A)\right),
\end{equation}
where $g_{s-rq,\beta}$ is defined in~\eqref{ggg}. Note first that, by Proposition~\ref{prop:properties of the ln-measure}, our choice of $\zeta$ is admissible, namely, it satisfies the four conditions
\eqref{eq:property of zeta1}, \eqref{eq:property of zeta}, \eqref{eq:property of zeta3}, and \eqref{eq:property of zeta5}, which are required in order to apply Theorem~\ref{thm:Besov constant controls double average integral}. We now turn to the verification of condition~\eqref{eq:Besov constant is more that integral of lower limiting double integralgensum}. There exists $n_0\in \N$ such that $|n-2\ln 2|\geq \tfrac{1}{2}n$ for every $n\geq n_0$. Thus, for every $n\geq n_0$, we have by \eqref{eq:data for using theorem}

\begin{multline}
\label{eq:calculation of zeta}
\frac{\eta_n}{(2\e_n)^s}\zeta\big(B_{(2\e_n)}(x)\big)
=\frac{n^\alpha}{(2\e_n)^s}g_{s-rq,\beta}(4\e_n)=\frac{n^\alpha}{(2\e_n)^s}\frac{(4\e_n)^{s-rq}}{|\ln (4\e_n)|^{\beta}}
\\
=\frac{n^{\alpha}\,2^{s-rq}}{|n-2\ln 2|^\beta}(2\e_n)^{-rq}
\leq \frac{2^{s-rq+\beta}}{n^{\beta-\alpha}}(2\e_n)^{-rq}.
\end{multline}
Therefore, using \eqref{eq:calculation of zeta} and that $\beta-\alpha>1$, we obtain
\begin{multline}
\label{eq:Besov constant is more that integral of lower limiting double integralgensumolilohgjg}
\sum\limits_{n=n_0}^{\infty}\left(\eta_n\int_X\frac{1}{(2\e_n)^s}\zeta\big(B_{(2\e_n)}(x)\big)\Bigg\{\fint_{B_{(2\e_n)}(x)} W(z,x)d\mu(z)\Bigg\}d\mu(x)\right)
\\
=\sum\limits_{n=n_0}^{\infty}\frac{2^{s-rq+\beta}}{n^{\beta-\alpha}}\int_X\Bigg\{\fint_{B_{(2\e_n)}(x)} \frac{\rho\left(u(z),u(x)\right)^q}{(2\e_n)^{rq}}d\mu(z)\Bigg\}d\mu(x)
\\
\leq 2^{s-rq+\beta}
\|u\|^q_{B^r_{q,\infty}(X,\mu;Y)}\left(\sum\limits_{n=n_0}^{\infty}\frac{1}{n^{\beta-\alpha}}\right)
<\infty\,.
\end{multline}
Thus we can apply the
third part of
Theorem \ref{thm:Besov constant controls double average integral} to deduce that the set $A$ defined as
\begin{equation}
\label{eq:def-D for general rytry}
A:=\Set{x\in X}[\limsup_{n \to \infty} n^\alpha\fint_{B_{\e_n}(x)} \fint_{B_{\e_n}(x)}
\rho\left(u(z),u(y)\right)^q\, d\mu(z)\, d\mu(y)>0],
\end{equation}
satisfies
\begin{equation}
\label{eq:logarithmic measure of A is zero}
\Lambda^{s-rq,\beta}(A)=0.
\end{equation} 
Pay attention to the fact that the third part of Theorem~\ref{thm:Besov constant controls double average integral} yields the negligibility of a set with respect to a general measure $\psi$ arising from Carath\'eodory’s construction, whereas here we take $\psi=\Lambda^{s-rq,\beta}$. Note that the logarithmic measure $\Lambda^{s-rq,\beta}$, which is defined via a specific choice of the function $\zeta$, allows us to ensure the finiteness of the expression on the left-hand side of \eqref{eq:Besov constant is more that integral of lower limiting double integralgensumolilohgjg}, by using the Besov property together with the finiteness of the series $\sum_{n=1}^\infty\frac{1}{n^{\beta-\alpha}}$, where $\beta-\alpha>1$.

Since $\mu$ is an Ahlfors measure, we have 
\begin{equation}
\label{eq:Ahlfors property with power 2}
C^{-2}(\e_n)^{-2s}\leq \mu(B_{\e_n}(x))^{-2}\leq c^{-2}(\e_n)^{-2s}.
\end{equation}
Thus, \eqref{eq:Ahlfors property with power 2} and \eqref{eq:def-D for general rytry} give
\begin{equation}
\label{eq:def-D for general rytry1}
A=\Set{x\in X}[\limsup_{n \to \infty} \frac{n^\alpha}{\e_n^{2s}}F(B_{\e_n}(x))>0],\,F(B):=\int_{B} \int_{B}
\rho\left(u(z),u(y)\right)^q\, d\mu(z)\, d\mu(y);
\end{equation}
where the set function $F$ is defined for every Borel set $B$. 
By Lemma \ref{lem:D equals Dseq} with $\xi=2s$, we get
\begin{equation}
\label{eq:def-D for general rklkkkl4}
A=\Set{x\in X}[\limsup_{\e \to 0^+} \left(\frac{|\ln\e|^\alpha}{\e^{2s}} F(B_{\e}(x))\right)>0],
\end{equation}
and, again using the fact that $\mu$ is Ahlfors regular, we obtain from \eqref{eq:def-D for general rklkkkl4} and \eqref{def: def D alpha}
\begin{equation}
\label{eq:def-D for general rklkkkl}
A=\Set{x\in X}[\limsup_{\e \to 0^+} \left(|\ln\e|^\alpha\fint_{B_{\e}(x)} \fint_{B_{\e}(x)}
\rho\left(u(z),u(y)\right)^q\, d\mu(z)\, d\mu(y)\right)>0]=D_\alpha.
\end{equation}
By \eqref{eq:logarithmic measure of A is zero} and \eqref{eq:def-D for general rklkkkl}, we deduce \eqref{eq: D alpha is negligible w.r.t logarithmic measure}.
Finally, by the definition of the logarithmic dimension, Definition \ref{def:Hausdorff dimension and logarithmic Hausdorff dimension}, we have
$\Lambda_{\dim} D_{\alpha}\leq s-rq$, and thus, by Proposition~\ref{prop:equality between Hausdorff-D}, we conclude that
\[
\mathcal{H}_{\dim} D_\alpha=\Lambda_{\dim} D_{\alpha}\leq s-rq.
\]
This completes the proof of the lemma.
\end{proof}

\begin{lemma}
\label{lem:sum of metrics}
Let $(Y,\rho)$ be a metric space. Let $\{v_n\}_{n\in\N}$ be any sequence such that 
\begin{equation}
\sum_{n=1}^\infty\rho(v_n,v_{n+1})<\infty.
\end{equation}
Then $\{v_n\}_{n\in\N}$ is a Cauchy sequence.
\end{lemma}

\begin{proof}
Let $\e>0$. There exists $M\in\N$ sufficiently large such that
\[
\sum_{n=M}^\infty\rho(v_n,v_{n+1})<\e.
\]
Therefore, for every $j\geq M$ and every $m\in\N$, by the triangle inequality we obtain
\begin{equation}
\rho(v_j,v_{j+m})
\leq \sum_{n=j}^{j+m-1}\rho(v_n,v_{n+1})
\leq \sum_{n=M}^{\infty}\rho(v_n,v_{n+1})
<\e.
\end{equation}
Thus, $\{v_n\}_{n\in\N}$ is Cauchy.
\end{proof}

The following theorem shows that for $\beta>q+1$, $\Lambda^{s-rq,\beta}$–almost every point is a Lebesgue point of a Besov function in the class $B^r_{q,\infty}$. In other words, the exceptional set of non–-Lebesgue points has $\Lambda^{s-rq,\beta}$–measure zero. Consequently, its logarithmic Hausdorff dimension is at most $s-rq$, and therefore its (usual) Hausdorff dimension is also at most $s-rq$, that proves Theorem~\ref{cor:Hasdorff dimension of D is at most s-rq,intro}. More precisely:

\begin{theorem}[Lebesgue points of Besov functions $B^r_{q,\infty}$]
\label{thm:varphi-Hasdorff measure of D}
Let $X$ be a metric space, and let $\mu$ be an Ahlfors $s$-regular measure on $X$ for some $s\in (0,\infty)$. Let $0\leq r\leq 1$, $0<q<\infty$ be such that $rq\leq s$. Let $(Y,\rho)$ be a complete metric space, and let $u\in B^r_{q,\infty}(X,\mu;Y)$. Let us denote 
\begin{equation}
\label{eq:def of N}
\mathcal{N}:=\Set{x\in X}[\limsup\limits_{\e\to 0^+}\fint_{B_{\e}(x)}
\rho\left(u(z),y\right)^q\, d\mu(z)>0\quad \forall y\in Y].
\end{equation}
If $\beta>q+1$, then
\begin{equation}
\Lambda^{s-rq,\beta}(\mathcal{N})=0.
\end{equation} 
In particular, for every $x\in X\setminus \mathcal{N}$, there exists $u^*(x)\in Y$, such that 
\begin{equation}
\label{eq:def-D for general rLebesgue}
\lim_{\varepsilon \to 0^+} \left(\fint_{B_{\varepsilon}(x)}
\rho\left(u(z),u^*(x)\right)^q\, d\mu(z)\right) = 0.
\end{equation}
Therefore, almost every point is a Lebesgue point of $u$ with respect to the measure $\Lambda^{s-rq,\beta}$. In particular, the exceptional set $\mathcal{N}$ of points that are not Lebesgue points has Hausdorff dimension at most $s-rq$.
\end{theorem}

\begin{proof}
Assume that $\beta>q+1$. Then, there exists $\alpha>q$ with $\beta>\alpha+1$. Let $D_\alpha$ be as in Lemma \ref{lem:varphi-Hasdorff measure of D}. We prove first that for every $x\in X\setminus D_\alpha$, there exists $u^*(x)\in Y$ such that
\begin{equation}
\label{eq:LP outside D alpha}
\lim_{\varepsilon \to 0^+} \left(\fint_{B_{\varepsilon}(x)}
\rho\left(u(z),u^*(x)\right)^q\, d\mu(z)\right) = 0.
\end{equation}
Let $x\in X\setminus D_\alpha$. Then, by \eqref{def: def D alpha with en} of 
Lemma \ref{lem:varphi-Hasdorff measure of D}, and raising to the power $1/q$, we have
\begin{equation}
\label{eq:def-D for general rytryjyj}
\lim_{n \to \infty} n^{\alpha/q}\left(\fint_{B_{\e_n}(x)} \fint_{B_{\e_n}(x)}
\rho\left(u(z),u(w)\right)^q\, d\mu(z)\, d\mu(w)\right)^{\frac{1}{q}} =0,\quad \e_n:=e^{-n}.
\end{equation}
In particular, using that $\alpha>q$, we get from \eqref{eq:def-D for general rytryjyj}
\begin{equation}
\label{eq:def-D for general rytryjyjsum}
\sum\limits_{n=1}^\infty\left(\fint_{B_{\e_n}(x)} \fint_{B_{\e_n}(x)}
\rho\left(u(z),u(w)\right)^q\, d\mu(z)\, d\mu(w)\right)^{\frac{1}{q}} <\infty.
\end{equation}
Obtaining the finiteness of the series in~\eqref{eq:def-D for general rytryjyjsum} is the main idea of the proof, as we shall see below. This finiteness is made possible by the presence of the factor $n^{\alpha/q}$, with $\alpha>q$, which arises from the expression of $D_\alpha$ in~\eqref{def: def D alpha with en}. 
This finiteness will enable us to construct, for a given point $x\notin D_\alpha$, a Cauchy sequence $\{v_n\}$ in $Y$ such that $u$ converges to its limit in the integral sense. We now proceed with the formal argument.

By monotonicity of integral 
\begin{multline}
\label{eq:def-D for general rytryjyjsumfh}
\fint_{B_{\e_n}(x)} \fint_{B_{\e_n}(x)}
\rho\left(u(z),u(w)\right)^q\, d\mu(z)\, d\mu(w)\geq \fint_{B_{\e_n}(x)} \left(\inf\limits_{v\in Y}\fint_{B_{\e_n}(x)}
\rho\left(u(z),v\right)^q\, d\mu(z)\right) d\mu(w)
\\
=\inf\limits_{v\in Y}\fint_{B_{\e_n}(x)}
\rho\left(u(z),v\right)^q\, d\mu(z).
\end{multline}
Therefore, by \eqref{eq:def-D for general rytryjyjsum} and \eqref{eq:def-D for general rytryjyjsumfh} we have
\begin{equation}
\label{eq:def-D for general rytryjyjsumjj}
\sum\limits_{n=1}^\infty\left(\inf\limits_{v\in Y}\fint_{B_{\e_n}(x)}
\rho\left(u(z),v\right)^q\, d\mu(z)\right)^{\frac{1}{q}} <\infty.
\end{equation}
By the definition of the infimum, for every $n\in\N$, there exists $v_n\in Y$ such that
\begin{equation}
\label{eq:def-D for general rytryjyjsumjjjjjjj}
\left(\fint_{B_{\e_n}(x)}
\rho\left(u(z),v_n\right)^q\, d\mu(z)\right)^{\frac{1}{q}}<\left(\inf\limits_{v\in Y}\fint_{B_{\e_n}(x)}
\rho\left(u(z),v\right)^q\, d\mu(z)\right)^{\frac{1}{q}}+2^{-n},
\end{equation}
and therefore, from \eqref{eq:def-D for general rytryjyjsumjj}, we get
\begin{equation}
\label{eq:def-D for general rytryjyjsumjjjkjjry1}
\sum\limits_{n=1}^\infty\left(\fint_{B_{\e_n}(x)}
\rho\left(u(z),v_n\right)^q\, d\mu(z)\right)^{\frac{1}{q}} <\infty.
\end{equation}
By the triangle inequality we have for $\mu$-almost every $z,w\in X$
\begin{equation}
\label{eq:triangle inequality fro rho v_n, v_n+1}
\rho(v_n,v_{n+1})\leq \rho\left(v_n,u(z)\right)+\rho\left(u(z),u(w)\right)+\rho\left(u(w),v_{n+1}\right).
\end{equation}
Therefore, from \eqref{eq:triangle inequality fro rho v_n, v_n+1}, we obtain, for a constant $C_q$ depending only on $q$,
\begin{equation}
\label{eq:triangle inequality fro rho v_n, v_n+1 in power q}
\rho(v_n,v_{n+1})^q\leq C_q\Big(\rho\left(v_n,u(z)\right)^q+\rho\left(u(z),u(w)\right)^q+\rho\left(u(w),v_{n+1}\right)^q\Big).
\end{equation}
Note that in case $0<q\leq 1$, we can choose $C_q=1$. Using \eqref{eq:triangle inequality fro rho v_n, v_n+1 in power q}, we have
\begin{multline}
\label{eq:expression for rho vn, vn+1}
\rho(v_n,v_{n+1})^q=\fint_{B_{\e_{n+1}}(x)}\left( \fint_{B_{\e_n}(x)}
\rho\left(v_n,v_{n+1}\right)^q\, d\mu(z)\right) d\mu(w)
\\
\leq C_q \fint_{B_{\e_n}(x)}
\rho\left(v_n,u(z)\right)^q\, d\mu(z)
+C_q\fint_{B_{\e_{n+1}}(x)}\left( \fint_{B_{\e_n}(x)}
\rho\left(u(z),u(w)\right)^q\, d\mu(z)\right) d\mu(w)
\\
+C_q\fint_{B_{\e_{n+1}}(x)}
\rho\left(u(w),v_{n+1}\right)^q\, d\mu(w).
\end{multline} 
Therefore, by \eqref{eq:expression for rho vn, vn+1} to the power $1/q$, \er{eq:def-D for general rytryjyjsumjjjkjjry1} and \er{eq:def-D for general rytryjyjsum} we have
\begin{equation} 
\label{eq:def-D for general rytryjyjsumjjjkjjrygrf}
\sum\limits_{n=1}^\infty\rho(v_n,v_{n+1}) <\infty.
\end{equation}
By Lemma \ref{lem:sum of metrics}, $\{v_n\}_{n\in\N}$ is a Cauchy sequence. Thus since $Y$ is a complete metric space, there exists $u^*(x)\in Y$ such that
\begin{equation} 
\label{eq:limit of rho vn, u*}
\lim\limits_{n\to\infty}\rho(v_n,u^*(x))=0.
\end{equation}
From \eqref{eq:def-D for general rytryjyjsumjjjkjjry1}, we have
\begin{equation}
\label{eq:average limit for vn}
\lim\limits_{n\to\infty}\fint_{B_{\e_n}(x)}
\rho\left(u(z),v_n\right)^q\, d\mu(z)=0.
\end{equation}
By the triangle inequality and raising to the power $q$, we get for $\mu$--almost every $z\in X$
\begin{equation}
\label{eq:triangle inequality for rho with u*}
\rho\left(u(z),u^*(x)\right)^q\leq C'_q\Big(\rho\left(u(z),v_n\right)^q+ \rho\left(v_n,u^*(x)\right)^q\Big).
\end{equation} 
Therefore, by \eqref{eq:limit of rho vn, u*},\eqref{eq:average limit for vn}, and \eqref{eq:triangle inequality for rho with u*}, we obtain \eqref{eq:LP outside D alpha}.

Now we prove that $\Lambda^{s-rq,\beta}(\mathcal{N})=0$ and $\mathcal{H}_{\dim}\mathcal{N}\leq s-rq$, where the set $\mathcal{N}$ is defined in \eqref{eq:def of N}. For every $y\in Y$ and Borel set $B\subset X$, set $H_y(B):=\int_{B}
\rho\left(u(z),y\right)^q\, d\mu(z)$.
Since $\mu$ is Ahlfors regular and using Lemma \ref{lem:D equals Dseq}, we get
\begin{multline}
\label{eq:representation of N through sequence}
\mathcal{N}=\bigcap_{y\in Y}\Set{x\in X}[\limsup\limits_{\e\to 0^+}\fint_{B_{\e}(x)}
\rho\left(u(z),y\right)^q\, d\mu(z)>0]
\\
=\bigcap_{y\in Y}\Set{x\in X}[\limsup\limits_{\e\to 0^+}\frac{1}{\e^s}H_y(B_\e(x))>0]
=\bigcap_{y\in Y}\Set{x\in X}[\limsup\limits_{n\to \infty}\frac{1}{\e_n^s}H_y(B_{\e_n}(x))>0]
\\
=\bigcap_{y\in Y}\Set{x\in X}[\limsup\limits_{n\to \infty}\fint_{B_{\e_n}(x)}
\rho\left(u(z),y\right)^q\, d\mu(z)>0].
\end{multline}
Thus, \eqref{eq:LP outside D alpha} and \eqref{eq:representation of N through sequence} give
\begin{equation}
\label{eq:N is a subset of D alpha}
X\setminus D_\alpha\subset X\setminus \mathcal{N}\quad \Longrightarrow \quad \mathcal{N}\subset D_\alpha.
\end{equation}
Therefore, by Lemma \ref{lem:varphi-Hasdorff measure of D}, since $\beta>\alpha+1$, we get 
\begin{equation}
\label{eq:logarithmic measure of N is negligible}
\Lambda^{s-rq,\beta}(\mathcal{N})\leq \Lambda^{s-rq,\beta}(D_\alpha)=0,
\end{equation}
and $\mathcal{H}_{\dim}D_\alpha\leq s-rq$.
By \eqref{eq:N is a subset of D alpha}, we get $\mathcal{H}_{\dim}\mathcal{N}\leq s-rq$. This completes the proof.
\end{proof}
\begin{remark}\label{inclusion of Sobolev to Besov}
The result of Theorem~\ref{thm:varphi-Hasdorff measure of D} holds immediately for functions in the class $W^{r,q}(X,\mu;Y)$, where $0<r<1$ and $0<q<\infty$, because $W^{r,q}(X,\mu;Y)\subset B^r_{q,\infty}(X,\mu;Y)$. This inclusion follows from the Ahlfors regularity of $\mu$:
\begin{multline}
\label{eq:Sobolev inside Besov}
\int_X \fint_{B_t(x)}
\frac{\rho(u(x),u(z))^q}{t^{rq}}
\, d\mu(z)\, d\mu(x)
\leq \frac{1}{c}\int_X \int_{B_t(x)}
\frac{\rho(u(x),u(z))^q}{t^{rq+s}}
\, d\mu(z)\, d\mu(x)
\\
\leq \frac{1}{c}\int_X \int_{B_t(x)}
\frac{\rho(u(x),u(z))^q}{d(x,z)^{rq+s}}
\, d\mu(z)\, d\mu(x)
\leq \frac{1}{c}\int_X \int_{X}
\frac{\rho(u(x),u(z))^q}{d(x,z)^{rq+s}}
\, d\mu(z)\, d\mu(x).
\end{multline}
Since the right-hand side of \eqref{eq:Sobolev inside Besov} does not depend on $t$, we obtain $W^{r,q}(X,\mu;Y)\subset B^r_{q,\infty}(X,\mu;Y)$.
\end{remark}

\end{document}